\documentclass[12pt]{report}
\usepackage[all]{xy}
\usepackage{amsfonts,amsmath,oldgerm,amssymb,amscd}
\newcommand{\ra}{\rightarrow}
\newcommand{\lra}{\longrightarrow}

\newcommand{\by}[1]{\stackrel{#1}{\ra}}
\newcommand{\lby}[1]{\stackrel{#1}{\lra}}

\newcommand{\surj}{\ra\!\!\!\ra}	\newcommand{\inj}{\hookrightarrow}

\newcommand{\ol}{\overline}		\newcommand{\wt}{\widetilde}
\newcommand{\iso}{\by \sim}

\newtheorem{theorem}{Theorem}[chapter]
\newtheorem{proposition}[theorem]{Proposition}
\newtheorem{lemma}[theorem]{Lemma}
\newtheorem{conjecture}[theorem]{Conjecture}
\newtheorem{corollary}[theorem]{Corollary}
\newtheorem{question}[theorem]{Question}

\newcommand{\ga}{\alpha}		\newcommand{\gb}{\beta}
			
\newcommand{\gf}{\varphi}		
		\newcommand{\gl}{\lambda}

\newcommand{\gt}{\theta}		
		\newcommand{\gw}{\wedge}

\newcommand{\gL}{\Lambda}		
		\newcommand{\gT}{\Theta}

\newcommand{\BA}{\mbox{$\mathbb A$}}	
\newcommand{\BC}{\mbox{$\mathbb C$}}

	\newcommand{\BZ}{\mbox{$\mathbb Z$}}

	\newcommand{\CJ}{\mbox{$\mathcal J$}}
\newcommand{\CK}{\mbox{$\mathcal K$}}

\newcommand{\CS}{\mbox{$\mathcal S$}}

\newcommand{\ma}{\mbox{$\mathfrak a$}}

\newcommand{\mm}{\mbox{$\mathfrak m$}}	
	\newcommand{\p}{\mbox{$\mathfrak p$}}
\newcommand{\mq}{\mbox{$\mathfrak q$}}

\newcommand{\op}{\mbox{$\oplus$}}	
	\newcommand{\Spec}{\mbox{\rm Spec\,}}
\newcommand{\hh}{\mbox{\rm ht\,}}	\newcommand{\Hom}{\mbox{\rm Hom}}
\newcommand{\diag}{\mbox{\rm diag\;}}  	
\newcommand{\Um}{\mbox{\rm Um}}		
		\newcommand{\ot}{\mbox{$\otimes$}}
\newcommand{\Aut}{\mbox{\rm Aut}}	\newcommand{\End}{\mbox{\rm End}}

\def\Com{\UseComputerModernTips}

\begin{document}
\thispagestyle{empty}

\begin{center}
	\vspace*{.4in}
	{\LARGE \bf Complete Intersection Ideals and a Question of Nori}
	\vspace*{.7in}
	
{\large  A Thesis }\\ 
	{\large  Submitted to the Faculty of Science } \\
	{\large of the  University of Mumbai }\\
{\large for the Degree of Doctor of Philosophy}
	\vspace*{.5in}
	
{\large by} \\

   	{\Large \bf  Manoj Kumar Keshari}
\vspace*{.7in}

{\large under the supervision of} \\

{\Large \bf Prof. S. M. Bhatwadekar}
\vspace*{1in}

	{\Large   Tata Institute of Fundamental Research}\\
	{\Large  Mumbai, April - 2004}
\end{center}
\newpage

\vspace*{.4in}
\begin{center}
{\Large \bf Acknowledgments} 
\end{center}
\vspace*{.2in}

I take this opportunity to express my gratitude and thanks to my
thesis supervisor Professor S.M. Bhatwadekar for his invaluable
guidance and help during my thesis work. I would like to thank
Dr. Raja Sridharan for many helpful discussions and for beautiful lectures
in commutative algebra. I would like to thank Professors R. Sridharan,
S. Ramanan, S.G. Dani, Ravi A. Rao, Riddhi Shah, A.J. Parameswaran and
P.A. Gastesi for their lectures which I attended and also for their help
during my 1st year. I would like to
thank the office staff of the School of Mathematics for their prompt,
efficient and friendly handling of all administrative formalities.

I would like to thank my seniors Holla, Siddharta, Pralay, Preeti and
Amalendu for many valuable discussions and help. I thank Ig, Funda,
Venky, Kavitha, Gagan, Vinod, Gulab, Surjeet, Tomas, Alex, Sanjeeb,
Pratik, Patta, Debu, Amitava, Rahul, Krishnan, Neel, Yeshpal, Sandhu,
Yogesh, Ashok, Rajeev, Lakshmi, Balaji, Bipul, Nandi, Shashi for
making my stay in TIFR enjoyable. I thank Meera, Dishant, Awadhesh,
Dubey, Ram, Amit, Ritumoni, Vishal, Shanta, Ananth, Amala, Rabeya,
Vivek, Aarati, Sriram, Narayanan, Anand, Deepankar, Vaibhav, Ashutosh,
Amarnath, Apoorva, Deepshikha, Dareush, Sumedha, Poonam, Mrinal, Toni
for their friendship.

Words cannot express my gratitude to my parents. I dedicate this
thesis to them.

\setcounter{chapter}{-1}

\tableofcontents
\newpage
\pagenumbering{arabic}


\chapter{Introduction}

Let $k$ be a field and let
$V$ be a closed irreducible sub-variety of $\BA^n_k$. 
Let $\p \subset k[X_1,\ldots,X_n]$ be the prime
ideal defining the variety $V$.  We say that $V$ is 
{\it ideal theoretically} generated by $d$ elements if the ideal $\p$ is
generated by $d$ elements and
$V$ is {\it set-theoretically} generated by $d$ elements if
there exist $f_1,\ldots,f_d \in \p$ such that $\sqrt
{(f_1,\ldots,f_d)}= \p$, i.e. the variety $V$ is an intersection of $d$
hyper-surfaces. 

By a classical result of Kronecker \cite{Kro}, any variety in $\BA^n_k$
is an intersection of $n+1$ hyper-surfaces.
Eisenbud and Evans \cite{EEE} and Storch \cite{St} independently 
showed that any variety in $\BA^n_k$
is an intersection of $n$ hyper-surfaces, thus improving the above result
of Kronecker. In other words, given a prime
ideal $\p$ of $k[X_1,\ldots,X_n]$, there exist $f_1,\ldots,f_n \in \p$
such that $\sqrt {(f_1,\ldots,f_n)}= \p$, i.e. any prime ideal
$p$ in $k[X_1,\ldots,X_n]$ is set-theoretically generated by
$n$ elements. 

In view of the above result of Eisenbud-Evans and Storch, it is
natural to ask: 
\medskip

\noindent{\it Question: Does there exists a positive integer $d$ such that 
$\mu(\p)\leq d$ for all prime ideals $\p$ in
$k[X_1,\ldots,X_n]$ ($k$: field), where $\mu(\p)$ 
denotes the minimal number of generators of $\p$?}

Since $k[X_1,\ldots,X_n]$ is a UFD, every prime ideal $\p$
of height $1$ is principle. Also, it is well known that every maximal
ideal of $k[X_1,\ldots,X_n]$ is generated by $n$ elements (see
\cite{Matsu}, Theorem
5.3). Therefore, it is enough to consider prime ideals of height
$1< \hh \p <n$.

A classical example of Macaulay (see \cite{Abh} for details) shows
that given any positive integer $r\geq 4$, there exists a height $2$ 
prime ideal $\p$ in $\BC[X_1,X_2,X_3]$ such that $\mu(\p) \geq r$.
Therefore, the set $\{\mu(\p) \,|\, \p :$ prime ideal in
$\BC[X_1,X_2,X_3]\}$ is not bounded above. The ring
$\BC[X_1,X_2,X_3]/\p$ is not regular for all prime ideals $\p$ in 
Macaulay's example. Therefore, one can ask the following modified question: 
\medskip

\noindent{\it Question: Does there exists a positive integer $d$ such
that $\mu(\p) \leq d$ for all prime ideal $\p \subset
k[X_1,\ldots,X_n]$ such that $k[X_1,\ldots, X_n]/\p$ is regular? }
\medskip

Forster \cite{Forster} gave an affirmative answer to the above
 question. More precisely, he proved the following:

\begin{theorem}
Let $\p \subset k[X_1,\ldots,X_n]$ be a prime ideal, where $k$ is a
field. Assume that $k[X_1,\ldots,X_n]/\p$ is regular. Then $\p/\p^2$ 
as $k[X_1,\ldots,X_n]/\p$-module is generated by $n$ elements and
hence $\p$ is generated by $n+1$ elements.
\end{theorem} 

Since $k[X_1,\ldots,X_n]$ is regular, if $\p\subset k[X_1,\ldots,X_n]$
is a prime ideal such that $k[X_1,\ldots,X_n]/\p$ is regular, then $\p$
is locally generated by $\hh \p$ elements. Now, it follows from the
following theorem (\ref{F}) of Forster \cite{Forster} (with
$A=k[X_1,\ldots,X_n]$ and $M=\p/\p^2$) that $\p/\p^2$ is
generated by $n$ elements. 

\begin{theorem}\label{F}
Let $A$ be a Noetherian ring and let
$M$ be a finitely generated $A$-module. Then $M$ is generated by
sup $\{ \mu(M_{\mq})+\dim (A/\mq) : \mq\in \Spec A \}$ elements.
\end{theorem}

Moreover, after proving above result, Forster conjectured:

\begin{conjecture}
Let $\p\subset k[X_1,\ldots,X_n]$ be a prime ideal such that
$k[X_1,\ldots,X_n]/\p$ is regular. Then
$\p$ is generated by $n$ elements, i.e. $\mu(\p)\leq n$.
\end{conjecture}

Note that to prove Forster's conjecture, we can assume that 
$\hh \p \geq 2$. 

Abhyankar \cite{Abh1} (in the case $k$ is algebraically closed) and Murthy
\cite{Mu} independently settled the
case $n=3$ (the first non-trivial case) 
of Forster's conjecture. More precisely, they proved that
if $\p$ is a prime ideal of $k[X_1,X_2,X_3]$ such that
$k[X_1,X_2,X_3]/\p$ is regular, then $\p$ is generated by $3$
elements.  General case of Forster's conjecture was settled by Sathaye
\cite{Sathaye} (in the case $k$ is an infinite field) and Mohan Kumar
\cite{MK-3} independently. 

To prove Forster's conjecture, Mohan Kumar
proved the following more general result (proof of this is implicit in
(\cite{MK-3}, Theorem 5)).

\begin{theorem}\label{<}
Let $A$ be a commutative Noetherian ring and let $I$ be an ideal of
$A[T]$ such that $I/I^2$ is generated by $n$ elements. Assume that $n
\geq \dim(A[T]/I) +2 $. If $I$ contains a monic polynomial, then $I$
is a surjective image of a projective $A[T]$-module of rank $n$ with
trivial determinant.
\end{theorem}

Forster's conjecture follows from above result of Mohan Kumar. For,
suppose $\p$ is a prime ideal of $k[X_1,\ldots,X_n]$ of height $\geq
2$ such that $\p/\p^2$ is generated by $n$ elements. Since $\hh \p
\geq 2$, $n\geq \dim (k[X_1,\ldots,X_n]/\p) +2$. Further, after a change
of variable, we can assume that $\p$ contains a monic polynomial in
the variable $X_n$. Hence, by Mohan Kumar's result (\ref{<}), $\p$ is
a surjective image of a projective $k[X_1,\ldots,X_n]$-module of rank
$n$. Since, by Quillen-Suslin result \cite{Quillen,Su}, every
projective $k[X_1,\ldots,X_n]$-module is free. Hence $\p$ is generated
by $n$ elements.

Subsequently, Mandal improved Mohan Kumar's result by showing that $I$
is generated by $n$ elements (\cite{Mandal-2}, Theorem 1.2). More
precisely, he proved the following result. 

\begin{theorem}
Let $A$ be a commutative Noetherian ring and let $I$ be an ideal of
$A[T]$ such that $I/I^2$ is generated by $n$ elements. Assume that $n
\geq \dim(A[T]/I) +2 $.  If $I$ contains a monic polynomial, then $I$
is generated by $n$ elements. In-fact, he proved that any $n$ generators
of $I/I^2$ can be lifted to $n$ generators of $I$. 
\end{theorem}

It is interesting to investigate the following:\\ 
{\it Question: In what generality the above result of
Mandal is valid?}
\medskip

Suppose that $A$ is the coordinate ring of the real 3-sphere
and $\mm$ is a real maximal ideal. Let $I = \mm A[T].$ Then, it is
easy to see that $\mu(I/I^2) = 3 = \dim (A[T]/I) +2.$ Since $\mm$ is
not generated by $3$ elements \cite{Bora}, $I$ can not be
generated by $3$ elements. Such examples show that the above result
of Mandal is not valid for an ideal $I$ not containing a monic
polynomial without further assumptions. 

Obviously, one such natural
assumption would be that $I(0)$ is generated by $n$ elements, where
$I(0)$ denotes the ideal $\{ f(0) : f(T)\in I \}$ of $A$. Even then,
as shown in (\cite{Bh-Raja-1}, Example 5.2) $I$ may not be generated
by $n$ elements.  Therefore, it is natural to ask: {\it what further
conditions are needed to conclude that $I$ is generated by $n$
elements?}  Towards this goal, motivated by a result from topology
(see Appendix by M. Nori in \cite{Mandal-3}), Nori posed the following
general question:
   
\begin{question} 
Let $A$ be a regular affine domain of dimension $d$ over an infinite
perfect field $k$ and
let $n$ be an integer such that $2n \geq d+3$.
Let $I$ be a prime ideal of $A[T]$ of height $n$
such that $A[T]/I$ and $A/I(0)$
are regular $k$-algebras. 
Let $P$ be a projective $A$-module of rank $n$ and let $\phi : P[T]
\surj I/(I^2T)$ be a surjection. Then, can we lift $\phi$ to a
surjection from $P[T]$ to $I$? 
\end{question}

Note that, giving a surjection $\phi : P[T] \surj I/(I^2T)$ is
equivalent to giving two surjections $\psi : P[T] \surj I/I^2$ and
$\ga : P \surj I(0)$ such that $\psi \ot A/I(0)=\ga \ot A/I(0) : P
\surj I(0)/I(0)^2$
(\cite{Bh-Raja-1}, Remark 3.9). 
\medskip

The {\bf main result} of this thesis (\ref{general}) gives an affirmative
answer to the above question of Nori. More precisely, we prove the
following (\cite{Bh-Manoj}, Theorem 4.13):

\begin{theorem}\label{,}
Let $k$ be an infinite perfect field and let 
$A$ be a regular domain of dimension $d$ which is 
essentially of finite type over 
$k$. Let $n$ be an integer such that $2n \geq
d+3$. Let $I\subset A[T]$ be an ideal of height $n$ and let $P$ be a 
projective $A$-module of rank $n$.
	Assume that we are given a surjection 
$$\phi : P[T]\surj I/(I^2T).$$ 
	Then, there exists a surjection
$$\Phi: P[T]\surj I$$ 
	such that $\Phi$ is a lift of $\phi$.

In particular, suppose $I/(I^2T)$ is
generated by $n$ elements. Then $I$ is generated by $n$ elements.

\end{theorem}

Prior to our theorem, the following partial results were obtained:

Mandal (\cite{Mandal-3}, Theorem 2.1) answered the question in affirmative 
in the case $I$ contains 
a monic polynomial even without any smoothness condition. An example
is given in the case $d=n=3$ (see \cite{Bh-Raja-1},
 Example 6.4) which shows that
the question does not have an affirmative answer if we do not assume
that $I$ contains a monic polynomial and  drop
the assumption that $A$ is smooth.

Mandal and Varma (\cite{Mandal-Varma}, Theorem 4) settled the
question, where $A$ is a regular $k$-spot (i.e. a local ring of a
regular affine $k$-algebra).  Subsequently, Bhatwadekar and Raja
Sridharan (\cite{Bh-Raja-1}, Theorem 3.8) answered the question in the
case $\dim (A[T]/I)=1$.
\medskip

Using the techniques developed to prove Theorem \ref{,}, we prove the
following result (\ref{5.4}). 

\begin{theorem}
Let $A$ be a commutative  Noetherian ring containing an infinite field 
and let $P$ be a projective $A[T]$-module of rank $r \geq (\dim A
+3)/2$ which is
extended from $A$. Assume 
that $P_{f(T)}$ has a unimodular element for some monic polynomial
$f(T)\in A[T]$. Then $P$ has a unimodular element. 
\end{theorem}

The above result gives a partial answer to the following question of
Roitman \cite{Roitman}. 

\begin{question}
Let $A$ be a commutative Noetherian ring and let $P$ be a projective
$A[T]$-module such that $P_{f(T)}$ has a unimodular element for some
monic polynomial $f(T)$. Then, does $P$ have a unimodular element?
\end{question}

The layout of the thesis is as follows: In chapter $1$, we recall some
basic definitions and state some well known results for later use. In
chapter $2$, we prove some basic results and Subtraction principle
which is the main ingredient for our main result. In chapter $3$, we
prove our main result. Chapter $4$ contains some applications of the
results proved in previous chapters.


\chapter{Preliminaries}

All rings considered in this thesis are commutative and Noetherian
with unity and all
modules are finitely generated. For a ring
$A$, the Jacobson radical of $A$ is denoted by $\CJ(A)$. We begin with
a few definitions and subsequently state some basic and useful results
without proof. For all the terms not defined here, we refer to
\cite{Matsu}. 

\begin{definition}
Let $A$ be a ring. The supremum of the lengths $r$, taken over all
strictly increasing chains 
	$\p_0 \subset \p_1 \subset \ldots\subset \p_r$ 
of prime ideals of $A$, is called the {\it Krull dimension} of
$A$ or simply the dimension of $A$, denoted by $\dim A$.

For a prime ideal $\p$ of $A$, the supremum of the lengths
$r$, taken over all strictly increasing chains 
	$\p_0 \subset \p_1 \subset \ldots\subset \p_r = \p$ 
of prime ideals of $A$, is called the
the height of $\p$, denoted by $\hh \p$. Note that for a
Noetherian ring $A$, $\hh \p < \infty$.

	For an ideal $I \subset A$, the infimum of the heights of
$\p$, taken over all prime ideals $\p \subset A$ such that $I \subset
\p$, is defined to be {\it height} of $I$, denoted by $\hh I$.
\end{definition}

\begin{remark}\label{lem1}
Let $I$ be an ideal of $A$. Then, it is clear from the definition that
$\dim (A/I) + \hh I \leq \dim A$.
\end{remark}

\begin{definition}
	An $A$-module $P$ is said to be {\it projective} if it
satisfies one of the following equivalent conditions:

	(i) Given $A$-modules $M,\,N$ and an $A$-linear surjective map
$\alpha : M \surj N$, the canonical map from Hom$_A(P,M)$ to
Hom$_A(P,N)$ sending $\theta$ to $\alpha \theta$ is surjective.

	(ii) Given an $A$-module $M$ and a surjective $A$-linear map
$\alpha : M \surj P$, there exists an $A$-linear map $\beta : P \ra M$
such that $\alpha \beta = 1_P$.

	(iii) There exists an $A$-module $Q$ such that $P \oplus Q
\simeq A^n$ for some positive integer $n$, i.e. $P \oplus Q$ is free.
\end{definition} 
\medskip

Now, we state the well known Nakayama Lemma.

\begin{lemma}\label{Nakayama}
	Let $A$ be a ring and let $M$ be a finitely
generated $A$-module.  Let $I\subset A$ be an ideal such that $I M
=M$.  Then, there exists $a\in I$ such that $(1+a)M = 0$.  In
particular, if $I \subset \CJ(A)$, then
$(1+a)$ is a unit and hence $M = 0$.
\end{lemma}

\begin{lemma} 
Let $I$ be an ideal of $A$ which is contained in the Jacobson
radical of $A$. Let $P,Q$ be projective $A$-modules.  If projective
$A/I$-modules
$P/IP$ and $Q/IQ$ are isomorphic, then $P$ and $Q$ are isomorphic as
$A$-modules.  
\end{lemma}

\begin{proof}
	Let $\ol \alpha : P/IP \iso Q/IQ$ be an isomorphism. 
Since $P$ is projective, $\ol \alpha$ can be lifted to an $A$-linear
map $\alpha : P \ra Q$. We claim that $\alpha$ is an isomorphism.

	Since $\ol \alpha$ is surjective, $Q = \alpha (P) + IQ$. Hence,
as $I\subset \CJ(A)$, by Nakayama lemma
(\ref{Nakayama}), 
we get $Q = \alpha(P)$. Hence $\alpha$ is surjective.

	Since $Q$ is projective, there exists an $A$-linear map 
	$\beta : Q \ra P$ such that $\alpha \beta = $ Id$_Q$. 
Let $\ol \beta : Q/IQ \ra P/IP$ be the map induced by $\beta$. 
Then,  we have $\ol \alpha \,\ol \beta = $ Id$_{Q/IQ}$. 
	As $\ol \alpha$ is an isomorphism, we get
that $\ol \beta$ is also an isomorphism and in particular $\ol \beta$
is surjective. Therefore $P = \beta(Q) +IP$. Hence, as before, we see
that $\beta$ is surjective. Now, injectivity of $\alpha$ follows from
the fact that $\alpha\beta = $ Id. 
$\hfill \square$
\end{proof}
\medskip

The following result is an immediate consequence of the above result.

\begin{corollary}\label{=}
	Let $A$ be a local ring. Then, every  projective
$A$-module is free.
\end{corollary}

\begin{definition}
For a ring $A$, $\Spec A$ denotes the set of all prime ideals of $A$. 
For an ideal $I\subset A$, we denote by
$V(I)$, the set of all prime ideals of $A$ containing $I$. For $f\in
A$, we denote by $D(f)$, the set of all prime ideals of $A$ not
containing the element $f$. The {\it Zariski topology} on Spec$\;(A)$
is the topology for which all the closed sets are of the form $V(I)$,
for some ideal $I$ of $A$ or equivalently the basic open sets are of
the form $D(f),\;f\in A$. 
\end{definition}

\begin{definition}
Let $P$ be a projective $A$-module. In view of (\ref{=}), we define
the rank function, rank$_P : \Spec A \ra \BZ$ by rank$_P(\mq) =$ rank
of the free $A_{\mq}$-module $P\otimes_A A_{\mq}$. If rank$_P$ is a
constant function taking the value $n$, then, we define the rank of
$P$ to be $n$ and denote it by rk $P$.
\end{definition}

\begin{remark}
	rank$_P$ is a continuous function (with the discrete topology
on $\BZ$ and Zariski topology on Spec $A$). Moreover, rank$_P$ is a
constant function for every finitely generated projective $A$-module
$P$ if $A$ has no non trivial idempotent elements.
\end{remark}

\begin{definition}
	Given a projective $A$-module $P$ and an element $p \in P$, we
define ${\cal O}_P(p) = \{ \alpha (p) \,|\,\alpha \in P^\ast\}$. We
say that $p$ is {\it unimodular} if ${\cal O}_P(p) = A$. The set of
all unimodular elements of $P$ is denoted by $\Um(P)$. 
If $P = A^n$, then we write $\Um_n(A)$ for $\Um(A^n)$. Note that ${\cal
O}_P(p)$ is an ideal of $A$ and $p\in P$ is a unimodular element if and
only if there exists $\ga\in P^*=\Hom_A(P,A)$ such that $\ga(p)=1$.

Let $P$ be a projective $A$-module of rank $n$. Let $\gw^n(P)$ denote
the $n^{th}$ exterior power of $P$. Then $\gw^n(P)$ is a projective
$A$-module of rank $1$ and is called the {\it determinant of $P$}. We
say determinant of $P$ is trivial if $\gw^n(P)=A$. 
\end{definition}
\medskip

Now, we state a classical result of Serre \cite{S1}.

\begin{theorem}\label{Serre}
Let $A$ be a ring with $\dim (A/\CJ(A))= d$. Then, any projective
$A$-module $P$ of rank $> d$ has a unimodular element. 
\end{theorem}   

The following is a classical result of Bass \cite{Bass}.

\begin{theorem}\label{1bass}
Let $A$ be a ring of dimension $d$ and let $P$ be a projective $A$-module 
of rank $>d$. Let $(p,a)\in \Um(P\op A)$. Then, there exists $q\in P$
such that $p+aq \in \Um(P)$. In particular,
 $E(P\op A)$ acts transitively on $\Um(P\op A)$. 
\end{theorem}

\begin{notation}
Let $A$ be a ring and let $A[T]$ be the polynomial algebra in one
variable $T$. We denote, by $A(T)$, the ring obtained from $A[T]$ by 
inverting all monic polynomials. For an ideal $I$ of $A[T]$ and $a\in
A$, $I(a)$ denotes the ideal $\{ f(a) : f(T)\in I  \}$ of $A$.

Let $P$ be a projective $A$-module. Then $P[T]$ denotes
the projective $A[T]$-module $P\ot_A A[T]$ and
$P(T)$ denotes the projective $A(T)$-module
$P[T]\ot_{A[T]} A(T)$.
\end{notation}

\begin{definition}
Let $B$ be a ring and let $P$ be a projective $B$-module. 
Given an element $\gf\in P^\ast$ and an element $p\in P$, we define an
endomorphism $\gf_p$ of $P$ as the composite $P\by \gf B\by p P$.
If $\gf(p)=0$, then ${\gf_p}^2=0$ and hence $1+\gf_p$ is a unipotent 
automorphism of $P$.

By a {\it transvection}, we mean an automorphism of $P$ of the form
$1+\gf_p$, where $\gf(p)=0$ and either $\gf$ is unimodular in $P^\ast$
or $p$ is unimodular in $P$. We denote by $E(P)$ the subgroup of
$\Aut(P)$ generated by all transvections of $P$. Note that $E(P)$ is a
normal subgroup of $\Aut(P)$. Also, an existence of a
transvection of $P$ pre-supposes that $P$ has a unimodular element.
\end{definition}

\begin{definition}\label{lem7}
Let $B$ be a ring and let $P$ be a projective $B$-module. An
automorphism $\sigma$ of $P$ is said to be {\it isotopic to identity},
if there exists an automorphism $\Phi(W)$ of the projective
$B[W]$-module $P[W]=P\ot B[W]$ such that $\Phi(0)$ is the identity
automorphism of $P$ and $\Phi(1)=\sigma$. Two elements 
$p_1, p_2 \in P$ are said to be {\it isotopically connected} if
there exists an automorphism $\sigma$ of $P$ such that 
$\sigma$ is isotopic to identity and $\sigma(p_1) = p_2.$
\end{definition}

\begin{remark}\label{lem8} 
Let $B$ be a ring and let $P$ be a projective $B$-module.
Let $\sigma$ be an automorphism of $P$ and let $\sigma^\ast$ be the
induced automorphism of $P^\ast$ defined by $\sigma^* (\alpha)=\alpha
\sigma$ for $\alpha \in P^*$. 

If $\sigma \in E(P)$, then $\sigma^* \in E(P^*).$
If $\sigma$ is isotopic to identity, then, so  is $\sigma^*.$ 

If $\sigma$ is unipotent, then it is
isotopic to identity. Therefore, any element of $E(P)$ is also isotopic to
identity. 

Now, suppose that $B = A[T]$ and $P= Q[T] = Q \ot_A A[T]$. Then,
since ${\rm End}_B (P)  = {\rm End}_A(Q)[T],$ we regard  
$\sigma$ as polynomial in $T$ with coefficients in ${\rm End}_A(Q)$, say
$\sigma = \theta(T)$. 
If $\theta(0)$ is the identity automorphism of
$Q$, then, since $\Phi(W) = \theta (WT)$ is an automorphism of 
$Q[T,W] = Q\ot_A A[T,W] = P\ot_B B[W]$, it follows that $\sigma$ is  
isotopic to identity.
\end{remark}
\medskip

The following lemma follows from the well known Quillen's Splitting
lemma (\cite{Quillen}, Lemma 1) and its proof is essentially contained
in (\cite{Quillen}, Theorem 1).

\begin{lemma}\label{isotopy}
Let $B$ be a ring and let $P$ be a projective $B$-module. Let $a,b\in
B$ be such that $Ba+Bb=B$. Let $\sigma$ be a $B_{ab}$-automorphism of
$P_{ab}$ which is isotopic to identity. Then $\sigma = \tau_a
\,\gt_b$, where $\tau$ is a $B_b$-automorphism of $P_b$ such that
$\tau=$ Id modulo the ideal $aB_b$ and $\gt$ is a $B_a$-automorphism
of $P_a$ such that $\gt=$ Id modulo the ideal $bB_a$.
\end{lemma}

The following result is due to Bhatwadekar and Roy 
(\cite{Bh-Roy}, Proposition 4.1) and is about lifting an automorphism
of a projective module.

\begin{proposition}\label{trans}
Let $B$ be a ring and let $P$ be a projective $B$-module. Let $I\subset
B$ be an ideal. Then, any transvection $\ol \Phi$ of $P/IP$ (i.e. $\ol
\Phi \in E(P/IP)$) can be lifted to an
automorphism $\Phi$ of $P$.
\end{proposition}

The following result is a consequence of a theorem of Eisenbud and Evans
as stated in (\cite{P}, p. 1420).

\begin{theorem}\label{EE}
Let $A$ be a ring and let $P$ be a projective $A$-module of rank
$r$. Let $(\alpha,a) \in (P^\ast \oplus A)$. Then, there exists an
element $\beta \in P^\ast$ such that $\hh I_a \geq r$, where
$I=(\alpha+a \gb)(P)$. In particular, if the ideal $(\ga(P),a)$ has
height $\geq r$, then $\hh I \geq r$. Further, if $(\ga(P),a)$ is an
ideal of height $\geq r$ and $I$ is a proper ideal of $A$, then $\hh I
= r$.
\end{theorem}

The following result is due to Lindel (\cite{Lindel-1}, Theorem 2.6).
      
\begin{theorem}\label{lindel}
Let $B$ be a ring of dimension $d$
and $A = B[T_1,\ldots,T_n]$. Let $P$ be a projective $A$-module of
rank $\geq$ max $(2,d+1)$. Then $E(P\oplus A)$ acts transitively on
the set of unimodular elements of $P\oplus A$.
\end{theorem}

Now, we quote a result of Mandal (\cite{Mandal-3}, Theorem 2.1).

\begin{theorem}\label{mand}
Let $A$ be a ring and let $I\subset A[T]$ be an ideal containing a
monic polynomial. Let $P$ be a projective $A$-module of rank
$n\geq \dim (A[T]/I) +2$. Let $\phi : P[T] \surj I/(I^2T)$ be a
surjection. Then $\phi$ can be lifted to a surjection $\Phi :P[T]
\surj I$.
\end{theorem}

\begin{definition}
Let $A$ be a  local ring of dimension $d$ with unique
maximal ideal $\mm$. If $\mm$ is generated by $d$ elements, then $A$
is said to be a {\it regular local ring}. A ring $B$ is
called {\it regular} if $B_{\mm}$ is a regular local ring for
every maximal ideal $\mm$ of $B$. A local ring $A$ is called a
$k$-{\it spot} if it is a localization of an affine $k$-algebra. 
\end{definition}
\medskip

The following result is due to Mandal and Varma 
(\cite{Mandal-Varma}, Theorem 4).

\begin{theorem}\label{M-V}
Let $A$ be a regular $k$-spot, where $k$ is an infinite perfect
field. Let $I\subset A[T]$ be an ideal of height $\geq 4$ and let $n$
be an integer such that $n\geq \dim (A[T]/I) +2$. Let $f_1,\ldots,f_n \in
I$ be such that $I=(f_1,\ldots,f_n)+(I^2T)$. Assume that  
$I(0)$ is a complete intersection ideal of $A$ of height $n$ or
$I(0) = A$. Then $I= (F_1,\ldots,F_n)$ with $F_i -f_i \in (I^2T)$. 
\end{theorem}

The following result is a variant of (\cite{Bh-1}, Proposition
3.1). We give a proof for the sake of completeness.

\begin{proposition}\label{m-phil}
Let $B$ be a ring and let $I\subset B$ be an ideal of height $n$. 
Let $f \in B$ be such that it is not a zero divisor modulo $I$. 
Let $P=P_1\op B$ be a projective $B$-module of rank $n$. Let 
$\alpha : P \ra I$ be a linear map such that the induced map
$\alpha_f : P_f \surj I_f$ is a surjection.
Then, there exists $\Psi \in E({P_f}^*)$ such that

(1) $\beta = \Psi(\alpha) \in P^*$  and 

(2) $\beta (P)$ is an ideal of
$B$ of height $n$ contained in $I$.
\end{proposition}

\begin{proof}
Note that, since $f$ is not a zero divisor modulo $I$ and 
$\alpha_f (P_f) = I_f$, if  
$\Delta$ is an automorphism of ${P_f}^*$ such that 
$ \delta = \Delta (\alpha) \in P^*$, then $\delta(P) \subset I$.  

Let $\CS$ be the set 
$\lbrace\, \Gamma \in E({P_f}^*) : \Gamma (\alpha) \in P^* \, \rbrace$.
Then $\CS \neq \varnothing$, since the identity automorphism of ${P_f}^*$
is an element of $\CS$. For $\Gamma \in \CS$, let $N(\Gamma)$ denote 
height of the ideal $\Gamma (\alpha)(P)$.
Then, in view of the above observation, it is enough to prove that there 
exists $\Psi \in \CS$ such that $N(\Psi) = n$. This is proved by showing
that for any $\Gamma \in \CS$ with $N(\Gamma) < n$, there exists
$\Gamma_1 \in \CS$ such that $N(\Gamma) < N(\Gamma_1)$.

Since $P = P_1 \oplus B$, we write $\alpha = (\gt,a)$, where 
$\gt \in {P_1}^*$ and $ a
\in B$.  Let $\Gamma \in \CS$ be such that $N(\Gamma) < n$. Let $\Gamma
((\gt,a)) =(\gb,b) \in {P_1}^* \oplus B$. Applying Eisenbud-Evans
theorem ($\ref{EE}$), there exists
$\phi \in {P_1}^\ast$ such that $\hh L_b \geq n-1$, where $L =
(\gb+b\phi)(P_1)$. It is easy to see that the automorphism $\Lambda$ of
${P_1}^* \oplus B$ defined by $\Lambda ((\delta, c)) = (\delta + c\phi,c)$
is a transvection of ${P_1}^* \oplus B $ and $\Lambda (\gb,b) =
(\gb+b\phi,b)$.  Hence $\Lambda\, \Gamma \in \CS$ and moreover
$N(\Gamma) = N(\Lambda \,\Gamma)$.  Therefore, if necessary, we can
replace $\Gamma$ by $\Lambda\, \Gamma$ and assume that if a prime ideal
$\p$ of $B$ contains $\gb(P_1)$ and does not contain $b$, then 
$\hh \p \geq n-1$. Now, we claim that $N(\Gamma) = \hh \gb(P_1)$.

We have $N(\Gamma) \leq n-1$. Since $N(\Gamma) = \hh (\gb(P_1),b)$, 
we have $\hh \gb(P_1) \leq N(\Gamma) \leq n-1$. Let $\p$ be a 
minimal prime ideal of $\gb(P_1)$ such that $\hh \p = \hh \gb(P_1)$.
If $b \notin \p$, then $\hh \p \geq n-1$. Hence, we have the
inequalities $n-1 \leq \hh \gb(P_1) \leq N(\Gamma)
\leq n-1$. This implies that $N(\Gamma) = \hh  \gb(P_1)= n-1$. 
If $b \in \p$, then $\hh \gb(P_1) = \hh \p \geq
\hh  (\gb(P_1),b) = N(\Gamma) \geq \hh \gb(P_1)$. This proves the claim.

Let $\CK$ denote the set of minimal prime ideals of
$\gb(P_1)$. Since $P_1$ is a projective $B$-module of rank $n-1$,
if $\p \in \CK$, then $\hh \p \leq n-1.$

Let $\CK_1 = \{ \p \in \CK : b \in \p \,\}$ and
let $\CK_2 = \CK - \CK_1$. Note that, since $\hh \gb(P_1) = 
\hh  (\gb(P_1),b)$, $\CK_1 \neq \varnothing$.  Moreover, every member
$\p$ of $\CK_1$ is a prime ideal of height $< n$ which contains 
$I_1 = (\gb(P_1),b).$ Therefore, since $(I_1)_f = I_f$ and $\hh I = n$, 
it follows that $f \in \p$ for all $\p \in \CK_1$.

 Since $\bigcap_{\p \in \CK_2}\p \not\subset \bigcup_{\p \in \CK_1}
\p$, there exists $x \in \bigcap_{\p \in \CK_2}\p$ such that $x \notin
\bigcup_{\p \in \CK_1} \p$. Since $f \in \p$ for all $\p \in \CK_1$, we
have $xf \in \bigcap_{\p \in \CK} \p$. This implies that $(xf)^r \in
\gb(P_1)$ for some positive integer $r$. 

Let $(x f)^r = \gb(q)$. As before, it is easy to see that the
automorphism $\Phi$ of ${P_1}^\ast \op B$ defined by $\Phi((\tau, d)) =
(\tau, d + \tau(q))$ is a transvection of ${P_1}^\ast \op B.$ Let $\Delta$
be an automorphism of $ {(P_1)_f}^* \oplus B_f $ defined by $\Delta (\eta,
c) = (\eta, f^rc)$.  Then, since $E({(P_1)_f}^\ast\op B_f)$ is a normal
subgroup of $GL({(P_1)_f}^\ast\op B_f)$, $\Phi_1 =\Delta^{-1}\, \Phi\,
\Delta$ is an element of $E({(P_1)_f}^\ast\op B_f).$ Moreover, $\Phi_1
((\gb,b)) = (\gb, b+x^r).$

Let $\Gamma_1 = \Phi_1 \,\Gamma.$ Then 
$\Gamma_1 (\alpha) = \Gamma_1 ((\gt,a)) = \Phi_1 ((\beta,b)) = (\gb, b+x^r).$
Therefore $\Gamma_1 \in \CS.$ Moreover, since $b+x^r$ does not belong to any
minimal prime ideal of $\gb(P_1),$  we have 
$N(\Gamma) = \hh \gb(P_1) < N(\Gamma_1) .$
This proves the result.  
$\hfill \square$
\end{proof}


\chapter{Subtraction Principle}

In this chapter, we prove ``Subtraction principle" (\ref{subtraction})
together with some other results for later use. Though these results are
technical in nature, they are the backbone for our main result
(\ref{general})  proved in this thesis.
We begin with the following easy lemma.

\begin{lemma}\label{010}
Let $B$ be a ring and let $I$ be an ideal of $B$. Let $K\subset I$ be
an ideal such that $I = K+I^2$. Then $I=K$ if and only if any maximal
ideal of $B$ containing $K$ contains $I$.
\end{lemma}

\begin{proof}
Since $I/K$ is an idempotent ideal of a Noetherian ring $B/K$ and
$I^2$ maps surjectively onto $I/K$, there exists an element $a\in
I^2$ such that $K+(a)=I$ and $a(1-a)\in K$. Therefore, $(1-a)I\subset
K$ and hence $I_{\mm}=K_{\mm}$ for every maximal ideal $\mm$ of $B$,
since any maximal
ideal of $B$ containing $K$ contains $I$. Hence $I=K$.
$\hfill \square$
\end{proof}

 
\begin{lemma}\label{002}
Let $B$ be a  ring and let $I\subset B$ be an ideal.
Let $I_1$ and $I_2$ be ideals of $B$ contained in $I$ such that
$I_2\subset I^2$ and $I_1+ I_2 = I$. Then $I=I_1 +(e)$ for some
$e\in I_2$ and $I_1 = I\cap I'$, where $I_2  +I'= B$.
\end{lemma}

\begin{proof}
Since $I/I_1$ is an idempotent ideal of a Noetherian ring $B/I_1$ and
$I_2$ maps surjectively onto $I/I_1$, there exists an element $a\in
I_2$ such that $I_1+(a)=I$ and $a(1-a)\in I_1$. The result follow by
taking $I'=I_1+(1-a)$.
$\hfill \square$
\end{proof}
\medskip

The proof of the following result uses the explicit completion of the
unimodular row $[a^2,b,c]$ given by Krusemeyer \cite{Kruse}.
 
\begin{lemma}\label{krusem}
Let $B$ be a ring and let $I=(c_1,c_2)$ be an ideal of $B$. Let $b\in B$ be
such that $I+(b)=B$ and let $r$ be a positive even integer.
Then $I=(e_1,e_2)$ with $c_1 - e_1 \in I^2$ and 
$b^r c_2-e_2\in I^2$.
\end{lemma}

\begin{proof}
Replacing $b$ by $b^{r/2}$, we can assume that $r=2$.
Since $b$ is a unit modulo $I=(c_1,c_2)$, it is unit modulo
$({c_1}^2,{c_2}^2)$. Let $1-bz=x'{c_1}^2+y'{c_2}^2 = xc_1+yc_2$, where
$x=x'c_1\in I$ and $y=y'c_2\in I$. The unimodular row $(z^2,c_1,c_2)$
has the following Krusemeyer completion (see \cite{Kruse})
 to an invertible matrix
$\Gamma$ given by
$$ \begin{pmatrix} z^2 & c_1 & c_2 \\ -c_1-2zy & y^2 & b-xy \\
		-c_2+2z x & -b-xy & x^2 \end{pmatrix}.$$ 
Let $\Theta : B^3 \surj I$ be a surjective map defined by 
$\gT(1,0,0)=0,~ \gT(0,1,0)=-c_2$ and $\gT(0,0,1)=c_1$. Then, since
$\Gamma$ is invertible and $\gT\,(z^2,c_1,c_2)=0$, it follows that
$I=(d_1,d_2)$, where 
 $d_1 = -y^2 c_2 + c_1 (b-xy)$ and $d_2 =
c_2(b+xy)+c_1x^2$. From the construction of elements $d_1$ and $d_2$, it
follows that  $d_1 - c_1b \in I^2$ and $d_2 -
c_2 b\in I^2$. Let $\Delta=\diag (z,b)\in M_2(B)$. Since diagonal matrices
of determinant $1$ are elementary, $\Delta \ot B/I \in E_2(B/I)$. Since 
the canonical map $E_2(B) \ra E_2(B/I)$ is surjective, there exists
$\Phi\in E_2(B)$ such that $\Delta \ot B/I = \Phi \ot B/I$. Let 
$[d_1,d_2] \,\Phi = [e_1,e_2]$. From the construction of $\Phi$, it follows
that $I=(e_1,e_2)$ with $e_1-c_1 \in I^2$ and $e_2-b^2 c_2\in I^2$.
  This proves the lemma.  
$\hfill \square$
\end{proof}

\begin{lemma}\label{square-lift}
Let $A$ be a ring and let $I$ be an ideal of $A$. Let $s\in A$
be such that $I+(s)=A$. Let $Q$ be a projective $A$-module
such that $Q/IQ$ is free and let $P= Q \op A^2$. 
Let $\Phi : P \surj I$ be a surjection. 
Let $r$ be  a positive  integer. Then, the map
$\Phi' = s^r\, \Phi : P \ra I$ induces a surjection
$\Phi' \ot A/I : P/IP \surj I/I^2$. Moreover if $r$ is even, then, the
surjection $\Phi' \ot A/I$
can be lifted to a surjection $\Psi : P \surj I$.
\end{lemma}

\begin{proof}
Since $I+(s)=A$ and $\Phi : P\surj I$ is a surjection, it is easy to see
that $\Phi' \ot A/I$ is a surjection from $P/IP$ to $I/I^2$.
 Now, we assume that $r=2l$.
 
Since $P = Q \op A^2$, we write $\Phi = (\phi,f_1,f_2).$ Let rank $Q/IQ=n-2$.
Let ``tilde'' denote reduction modulo $I$. Then, since $Q/IQ$ is free of
rank $n-2$, fixing a basis of $Q/IQ,$ we can write $\wt \Phi = (\wt
{k_1},\ldots, \wt {k_{n-2}},\wt {f_1}, \wt {f_2})$. Let 
$\gb = \diag (s^r,\ldots,s^r)$. Then $\wt \gb \in
\Aut(P/IP)$ and $\wt {\Phi^\prime} = \wt \Phi\, \wt \gb$. 
Since diagonal matrices of determinant $1$ are elementary,
we get $\wt \gb= \diag(1,\ldots,1,\wt {s^{nr}})\,\wt {\gb'}$, where $\wt
{\gb'}\in E(P/IP)$. By (\ref{trans}), $\wt {\gb'}$ can be lifted to an 
automorphism of $P$. Therefore, to prove the lemma, it is enough to show that 
the surjection 
$(\phi,f_1,s^{nr}f_2)\ot A/I : P/IP \surj I/I^2$
 can be lifted to a surjection $(\phi,g_1,g_2) : P\surj I$. Since $nr$ is
even, $s^{nr}={s_1}^2$. Therefore, replacing $s$ by $s_1$, we can
 assume that $nr=2$.

Let $K=\phi(Q)$ and let ``bar'' denote reduction modulo $K$.  Then
 $\ol I = (\ol {f_1},\ol {f_2})$. Applying (\ref{krusem}), we get $\ol I=
 (\ol {h_1},\ol {h_2})$ with $\ol {f_1} -\ol {h_1} \in \ol {I^2}$ and
$\ol {s^2 f_2} -\ol {h_2}\in \ol {I^2}$.
 Therefore, $I=(h_1,h_2)+K$, where $f_1 - h_1 =f'_1 +h'_1$ and
 $s^2 f_2 - h_2 = f'_2+h'_2$ for some $f'_1,f'_2\in
 I^2$ and $h'_1,h'_2\in K$.  Let $g_i= h_i+ h'_i$ for $i=1,2$. Then, we
 have $I =(g_1,g_2)+ K$ with $f_1-g_1\in I^2$ and $s^2 f_2-g_2\in
 I^2$. This proves the result.  
$\hfill \square$
\end{proof}

\begin{remark}
It will be interesting to know if the above result is valid without the
assumption that $Q/IQ$ is free.
\end{remark}

The following result is very crucial for our main result
(\ref{general}). 

\begin{lemma}\label{3.2}
Let $B$ be a ring and let $s,t\in B$ be
such that $Bs+Bt=B$. Let $I,L$ be ideals of $B$ such that $L\subset I^2$.
 Let $P$ be  a projective $B$-module and let $\phi : P\surj
I/L$ be a surjection. If $\phi\ot B_t$ can be lifted to a surjection $\Phi
:P_t\surj I_t$. Then $\phi$ can be lifted to
a surjection $\Psi : P\surj I/(sL)$.
\end{lemma}

\begin{proof}
Without loss of generality, we can assume that $t=1$ modulo the ideal $(s)$.
Let $l$ be a positive integer such that $t^l \Phi(P) \subset I$. 
Let $\Phi^\prime : P \ra I$ be a lift of $\phi$. Then, since $\Phi$
is a lift of $\phi_t,$ there exists an integer $r \geq l$ such that
$(t^r \Phi - t^r \Phi^\prime) (P) \subset L.$  Let $\Gamma = t^r \Phi$
and $K= \Gamma(P).$ Then, since $r \geq l$,  $K \subset I$. Since $K_t =
I_t$, we have $t^n I\subset K$ for some positive integer $n$. 
Since $1-t \in (s)$, $t^n=1-sx$ for some $x\in B$. Hence
$(1-sx)I\subset K$. Therefore, we have $K +sI = I.$

Let $t^r = 1- sa$ and let $\Theta = \Gamma + sa\Phi^\prime.$
Then $\Theta - \Phi^\prime = \Gamma - t^r \Phi^\prime.$
Therefore $(\Theta - \Phi^\prime)(P) \subset L$ and hence
$\Theta$ is also lift of $\phi.$ Therefore, $\Theta(P) + L=I$. Moreover, 
$\Theta (P) + sI = \Gamma (P) +sI = K+sI=I .$  
Write $I_1=\Theta(P)+sL$. Any maximal
ideal of $B$ containing $I_1$ contains $s$ or $L$ and hence contains $I$.
Moreover, since $L\subset I^2$, $I=I_1+I^2$.
Therefore, by (\ref{010}), $I=I_1$, i.e.
$\Theta(P) + sL = I.$ If $ \Gamma' : I \surj I/sL$  is a canonical surjection,
then putting $\Psi = \Gamma' \Theta $,  we are through.
$\hfill \square$
\end{proof}
\medskip

The following technical lemma is used in the proof of (\ref{monic})
which is very crucial for our main result (\ref{general}).

\begin{lemma}\label{sub}
Let $B$ be a ring and let $I_1,I_2$ be two comaximal ideals of
$B$. Let $P=P_1\op B$ be a projective $B$-module of rank $n$. Let
$\Phi : P \surj I_1$ and $\Psi : P\surj I_1\cap I_2$
be two surjections such that $\Phi\ot B/I_1 = \Psi \ot B/I_1$. Assume
that

$(1)~a=\Phi(0,1)$ is a non zero divisor modulo the ideal $\sqrt
{\Phi(P_1)}$.

$(2)~ n-1 > \dim (\ol B/\CJ(\ol B))$, where $\ol B =B/\Phi(P_1)$.

Let $L\subset {I_2}^2$ be an ideal such that
$\Phi(P_1)+L=B$.  Then, the surjection $\Psi : P \surj I_1\cap I_2$
induces a surjection $\ol \Psi : P \surj I_2/L$. Moreover, $\ol \Psi$
can be lifted to a surjection $\Lambda : P\surj I_2$.
\end{lemma}

\begin{proof}
Since $L+I_1=B$ (in fact $L+\Phi(P_1)=B$) , it is easy to see that 
$\Psi$ induces a surjection $\ol \Psi : P \surj I_2/L$.
 
Let $K=\Phi(P_1)$ and $S=1+K$. Then $S\cap L \neq \varnothing$.
Therefore, we have surjections $\Phi_S$ and $\Psi_S$ from $P_S$ to
$(I_1)_S$.

\paragraph{Claim:} There exists an automorphism $\Delta$ of $P_S$ such that
$\Delta^*(\Psi_S)=\Psi_S\, \Delta = \Phi_S$, where $\Delta^*$ is an
automorphism of ${P_S}^*$ induced from $\Delta$. 
\medskip

Assume the claim. Then, there exists $s=1+t\in S,~ t\in K$ such that
$\Delta \in \Aut (P_s)$ and
$\Psi_s\, \Delta = \Phi_s$. Since $S\cap
L \neq \varnothing$, we can assume that $s\in S\cap L$.

With respect to the decomposition $P=P_1 \op B$, we write $\Phi\in P^*$ as 
$(\Phi_1,a)$, where $\Phi_1 \in {P_1}^*$ and $a\in B$. Similarly, we
write $\Psi=(\Psi_1,b)$, where $\Psi_1 \in {P_1}^*$ and $b\in B$.
Let $pr : P_1\op B(=P) \surj B$ be the map defined by $pr (p_1,\wt
b)=\wt b$,
where $p_1\in P_1$ and $\wt b\in B$. 

Since $s\in L$, $(I_2)_s=B_s$ and therefore, we can regard $pr_s$
as a surjection from $(P_1)_s\op B_s$ to $(I_2)_s$.
Since $t\in K=\Phi_1(P_1)$, the element $(\Phi_1)_t\in 
{(P_1)_t}^\ast$ is a unimodular
element. Hence, there exists an element $\Gamma \in E((P_1)_{st} \op
B_{st})$ such that $\Gamma^*((\Phi_1,a)_{st}) = pr_{st}$
i.e. $(\Phi_t)_s \, \Gamma = (pr_s)_t$. Note that $\Psi_t$ is a
surjection from $P_t$ to $(I_2)_t$.

We also have $\Psi_s\, \Delta = \Phi_s$. Hence $(\Psi_s\, \Delta)_t
\, \Gamma = (pr_s)_t$. Let $\wt \Delta = \Delta_t\, \Gamma\, 
{\Delta_t}^{-1}$. Then, we have
$(\Psi_s)_t\, \wt \Delta= (\Psi_t)_s\, \wt \Delta
 =(pr_s)_t \,{\Delta_t}^{-1}$.  Since $\Gamma$ is an
element of $E(P_{st})$ which is a normal subgroup of $\Aut (P_{st})$,
 $\wt \Delta \in E(P_{st})$ and hence is isotopic to
identity, by (\ref{lem8}). Therefore, by (\ref{isotopy}), $\wt \Delta =
{\Delta''}_s\, {\Delta'}_t$, where $\Delta'$ is an automorphism of
$P_s$ such that $\Delta'=$ Id modulo $(t)$ and $\Delta''$ is an
automorphism of $P_t$ such that $\Delta''=$ Id modulo $(s)$.
 
Thus, we have surjections $(\Psi_t \,\Delta'') : P_t \surj (I_2)_t$ and 
$(pr_s\,\Delta^{-1}\, (\Delta')^{-1}) : P_s \surj (I_2)_s$ such that
$(\Psi_t\, \Delta'')_s = (pr_s\, \Delta^{-1}\, (\Delta')^{-1})_t$.
Therefore, they patch up to yield a surjection $\Lambda : P \surj
I_2$. Since
$s=1+t\in L$, the map $B\ra B/(s)$ factors through $B_t$. Since
$\Delta'' =$ Id modulo $(s)$, we have $\Lambda \ot B/L = \Psi\ot
B/L$. 

\paragraph{Proof of the claim:} 
Since $S=1+K$, $\ol B=B/K=B_S/K_S$ and $K_S \subset
\CJ(B_S)$. Therefore, $\ol B/\CJ(\ol B) = B_S/\CJ(B_S)$. Hence $\dim
B_S/\CJ(B_S) < n-1$.

	To simplify the notation,
we denote $B_S$ by $B$, $(P_1)_S$ by $P_1$ and $(I_1)_S$ by $I$. Then, we
have two surjections $\Phi =(\Phi_1,a)$ and $\Psi=(\Psi_1,b)$
from $P_1 \op B$ to $I$ such that $\Phi\ot B/I=\Psi\ot B/I$. 
Moreover, $\Phi_1(P_1) = K \subset \CJ (B)$ and rank $P_1(=n-1)>
\dim (\ol B/\CJ(\ol B))$, where $\ol B = B/K$.
Our aim is to show that there exists an automorphism $\Delta$ of
$P=P_1\op B$ 
such that $\Psi\, \Delta= \Phi$. 

Hence onward, we write an element $\sigma\in \End(P_1 \op B)$ in the
following matrix form
$$\sigma = \left(\begin{matrix} \alpha & p \\ \eta & d
	\end{matrix}\right),\; {\mbox {where}}~ \alpha \in \End(P_1),\,
	p\in P_1, \,\eta \in {P_1}^*~{\mbox {and}}~ d\in B. $$

	Note that, with this presentation of $\sigma\in \End(P)$, if
$\Theta=(\Theta_1,e)\in {P_1}^* \op B$, then $\sigma^* (\Theta) =
\Theta \sigma = (\Theta_1 \alpha + e \eta, \Theta_1(p)+ed)$.
Moreover, if $\sigma'\in \End(P)$ has a matrix representation 
$\sigma'=\left(\begin{smallmatrix}
	\beta & p_1 \\ \mu & f
	\end{smallmatrix}\right)$, then the endomorphism $\sigma'\sigma$
of $P$
has the matrix representation   
$$\sigma'\sigma = \left(\begin{matrix}
	\beta & p_1 \\ \mu & f
	\end{matrix}\right)
\left(\begin{matrix} \alpha & p \\ \eta & d
	\end{matrix}\right)
=
\left(\begin{matrix}
	\beta\alpha + \eta_{p_1} & \beta(p)+dp_1 \\ 
	\mu	\alpha + f\eta & \mu(p)+fd
	\end{matrix}\right),$$ where $\eta_{p_1} \in \End(P_1)$ is the
composite map $P_1 \lby \eta B \lby {p_1} P_1$. 
	 
Since $\Phi \ot B/I = \Psi \ot B/I : P\surj I/I^2$, $\Phi-\Psi : P \ra
I^2$. Since $\Phi : P \surj I$, $\Phi(IP)=I^2$. Hence, there
exists
$\eta : P \ra IP$ such that $\Phi \eta = \Phi-\Psi$ (since $P$ is
projective). Write $\Gamma =
Id-\eta$. Then $\Gamma \in \End(P)$ is
identity modulo the ideal $I$ and
 $\Phi\Gamma=\Psi$. Similarly, there exist
$\Gamma' \in \End(P)$ which is identity modulo the ideal $I$ such that
$\Psi\,\Gamma'=\Phi$.
Let 
$$\Gamma = \left(\begin{matrix}
	\gamma & q \\ \zeta & c
	\end{matrix}\right),~ 
\Gamma' = \left(\begin{matrix}
	\gamma' & q' \\ \zeta' & c'
	\end{matrix}\right)$$ be the matrix representation of $\Gamma$
and $\Gamma'$, 
where $\gamma,\gamma' \in
 \End(P_1)$, $q,q' \in P_1$, $\zeta,\zeta'\in {P_1}^\ast$ and $c,c'\in B$.
Then 
$$\Gamma\,\Gamma' = \left(\begin{matrix}
	\gamma\gamma' + {\zeta'}_q & \gamma (q') + c' q \\
	\zeta \gamma' + c \zeta' & \zeta (q') + c c'
	\end{matrix} \right). $$

Since $\Phi\,\Gamma\,\Gamma' = \Phi$ ($\Phi \Gamma=\Psi$ and
 $\Psi\Gamma'=\Phi$) and $\Phi = (\Phi_1,a)$, we get
 $\Phi_1(\gamma (q') + c'q) + a(\zeta (q') + c c') =a$. 
Hence $a(1-\zeta (q') - c c') \in K$. Since, by hypothesis, no minimal 
prime ideal of	$K$ contains $a$, we have
$(1-\zeta (q') - c c') \in \sqrt K$, i.e. $(\zeta (q') + c c')+\sqrt
	K = B$. But $K\subset \CJ(B)$ and hence 
$(\zeta (q') + c c')=B$, i.e. the element $\zeta(q')+cc' \in B^*$.
 Therefore  $(\zeta,c)\in P^*$ is a
 unimodular element. Note that, since $\Gamma$ is an endomorphism of
$P$ which is identity modulo $I$, $(\zeta,c)=(0,1)$ modulo $I$. Now, we
show that there exists an automorphism $\Delta_1$ of $P$ such that 
$(1)~ (\zeta,c)\,\Delta_1=(0,1)$ and $(2)~ \Delta_1$ is an identity
automorphism of $P$ modulo $I$. 

Let ``bar'' denote reduction modulo $K$. Since $\dim (\ol B/\CJ(\ol B))
 < n-1$, by a classical result of Bass (\ref{1bass}), there exists 
 $\zeta_1 \in {P_1}^\ast$ such
that 
$(\ol {\zeta+c\, \zeta_1})$ is a unimodular element of $\ol {{P_1}^\ast}$.
 But then, since $K\subset \CJ(B)$, $\zeta+c\,\zeta_1$ is a
unimodular element of ${P_1}^\ast$.
Let $q_1\in P_1$ be such that $(\zeta+c\,\zeta_1)(q_1)=1$. 
 Let 
$$\gf_1 = \left(\begin{matrix}
	1 & 0 \\ \zeta_1  & 1
	\end{matrix} \right),~
\gf_2 = \left(\begin{matrix}
	1 & (1-c)\,q_1 \\ 0 & 1
	\end{matrix}\right),~
\gf_3 = \left(\begin{matrix}
	1& 0  \\ - (\zeta +c\zeta_1)  & 1
	\end{matrix} \right) .$$

Write $ \Delta_1 = \gf_1\,\gf_2 \, \gf_3 .$ 
	 Since $(\zeta,c)=(0,1)$ modulo $I$, from the construction, it
follows that $\Delta_1$ is an automorphism of $P=P_1\op B$ which is
identity modulo $I$. Moreover, it is easy to see that $(\zeta,c)
\Delta_1 = (0,1)$. Therefore, we have 
$$\Gamma\,\Delta_1 = \left(\begin{matrix}
	\gamma_1 & q_2 \\ 0 & 1
	\end{matrix}\right),$$ for some $\gamma_1\in \End(P_1)$ and
	$q_2\in P_1$.
Since both $\Gamma$ and $\Delta_1$ are identity modulo $I$,
$\gamma_1$ is an endomorphism of $P_1$ which is identity modulo $I$
and $q_2 \in IP_1$. Let
$\Delta_2 =\left(\begin{matrix}
	1 & -q_2 \\ 0 & 1
	\end{matrix}\right).$
Then, $\Delta_2$ is an automorphism of $P_1\op B$ which is
identity modulo $I$. Moreover, 
$$\Delta=\Delta_2\, \Gamma\, \Delta_1 = \left(\begin{matrix}
	\gamma_1 & 0 \\ 0 & 1
	\end{matrix}\right).$$

Let $\wt a = \Phi_1(q_2) + a$. Then $\Phi\, {\Delta_2}^{-1} = ( \Phi_1,
\wt a)$ and hence $K +(\wt a) = I$. Moreover, $( \Phi_1, \wt
a)\, \Delta =  (\Phi_1\gamma_1, \wt a)=\Psi\,\Delta_1$ (since
$(\Phi_1,\wt a) \Delta = (\Phi_1,\wt a) \Delta_2 \Gamma \Delta_1 = \Phi
\Gamma \Delta_1 =\Psi \Delta_1$). Let $\wt \Psi_1 = \Phi_1\,\gamma_1$. 
Therefore, to complete the
proof (of the claim), it is enough to show that the surjections  
$\wt \Phi = ( \Phi_1, \wt a)$ and
$\wt \Psi = (\wt \Psi_1, \wt a)$  from $P$ to $I$ are connected by an
automorphism of $P$. Note that $\Delta \in \End(P)$.

Since $\gamma_1 \in \End(P_1)$ is identity
 modulo $I$, $(1-\gamma_1)(P_1) \subset IP_1$.
Since $P_1$ is a projective
$B$-module, we have $\Hom(P_1,IP_1)=I\,\Hom(P_1,P_1)$. 
Hence $1-\gamma_1 = \sum
b_i\beta_i$, where $\beta_i \in \End(P_1)$ and $b_i\in I$. Let $b_i =
c_i+d_i \wt a$, where $c_i\in K$ and $d_i\in B$. Then $1-\gamma_1 =
\sum c_i\beta_i + \wt a\sum d_i\beta_i$. Hence $\gamma_1 =
\theta + \wt a\theta'$, where $\theta = 1-\sum c_i \beta_i$ and
$\theta' = -\sum d_i \beta_i$. Since determinant of $\theta$ is $1+x$ for some
$x\in K \subset \CJ(B)$, $\theta$ is an automorphism of $P_1$. 

	We have $\wt \Psi_1 =
 \Phi_1 \gamma_1 =  \Phi_1\gt + \wt a \Phi_1 \theta'$.
  Let $\gL =
\left(\begin{matrix} \theta & 0 \\  \Phi_1 \theta' & 1
\end{matrix}\right).$ 
	Then  $( \Phi_1,\wt a)\gL = (\wt \Psi_1,
\wt a)$ and $\gL$ is an automorphism of $P$. This proves the result.
$\hfill \square$
\end{proof}
\newpage

Now, we will prove the main result of this chapter which is labeled
as ``Subtraction Principle''. This result is very important for the
proof of our main result and is also used crucially to prove other
results of the next chapter.

\begin{theorem}\label{subtraction}
Let $B$ be a  ring of dimension $d$ and let $I_1, I_2 \subset
B$ be two comaximal ideals of height $n$, where $2n \geq d+3$. Let
$P=P_1\op B$ be a projective $B$-module of rank $n$. Let $\Gamma : P\surj
I_1$ and $\Theta : P\surj I_1\cap I_2$ be two surjections such that 
$\Gamma \ot B/I_1
 =\Theta \ot B/I_1$. Then, there exists a surjection $\Psi : P\surj I_2$
such that $\Psi \ot B/I_2 = \Theta \ot B/I_2$.   
\end{theorem}

\begin{proof}
Let $\Gamma = (\Gamma_1,a)$. Let ``bar'' denote reduction modulo $I_2$. 
 Then $\ol \Gamma =(\ol
{\Gamma_1},\ol a)$ is a unimodular element of $\ol {P^*}$. Since, by
 (\ref{lem1}),  $\dim
(B/I_2) \leq \dim B - \hh I_2 =d-n <n=$ rank $\ol {P_1}$, 
by Bass' result (\ref{1bass}), there exists $\Theta_1 \in
{P_1}^*$ such that $\ol {\Gamma_1} + \ol {a^2 \Theta_1}$ is a
unimodular element of $\ol {{P_1}^*}$.
Therefore, replacing $\Gamma_1$ by $\Gamma_1+a^2 \Theta_1$, we can
assume that $\Gamma_1(P_1)=K$ is comaximal with $I_2$. Moreover,
using similar arguments, one can assume that height of $K$ is $n-1$
and therefore, $\dim (B/K)\leq d-(n-1)\leq n-2$.
 Since $K$ is a surjective image of
$P_1$ (a projective $B$-module of rank $n-1$), every minimal prime
ideal of $K$ has height $n-1$. Hence, since 
$I_1=K+(a)$ is an ideal of height $n$, $a$ is a non-zero divisor
modulo the ideal $\sqrt K$. Therefore, by 
(\ref{sub}), there exists a surjection $\Psi : P \surj I_2$ which is
a lift of $\Theta \ot B/I_2$. This proves the result.  
$\hfill \square$
\end{proof}

\begin{remark}
The above theorem has been already proved in the following cases.

$(1)$ In the case $P$ is free (\cite{Bh-Raja-5}, Proposition 3.2). 

$(2)$ For $n=d$ (\cite{Bh-Raja-3}, Theorem 3.3). 

Our approach is different from that of (\cite{Bh-Raja-5} and
\cite{Bh-Raja-3}) and we believe is of some independent interest.
\end{remark}


\chapter{Main Theorem}

In this chapter, we prove our main result (Theorem \ref{general}).  We
begin with the following lemma which is easy to prove (see
\cite{Gopal}, Proposition 1, p. 206). 

\begin{lemma}\label{555}
Let $A$ be a ring and $\p_1\subsetneqq \p_2\subsetneqq \p_3$ be a
chain of prime ideals of $A[T]$. Then, we can not have $(\p_1\cap A) =
(\p_2\cap A)= (\p_3\cap A)$.
\end{lemma}

\begin{lemma}\label{ht}
Let $A$ be a ring and let $I\subset A[T]$ be an ideal of height
$k$. Then $\hh (I\cap A) \geq k-1$.
\end{lemma}

\begin{proof}
First, we assume that $I=\p$ is a prime ideal. Then, we claim that 
$$\hh \p = \left\{ \aligned \hh (\p\cap A) ~~~~{\rm if} ~\p = (\p\cap A)[T] \\
\hh (\p\cap A)+1 ~~~{\rm if}~ \p \supsetneqq (\p\cap A)[T] \endaligned
\right.$$  

Any prime chain $\mq_0 \subsetneqq \ldots \subsetneqq \mq_r \subsetneqq
(\p\cap A)$ in $A$ extends to a prime chain $\mq_0[T] \subsetneqq
\ldots \subsetneqq \mq_r[T] \subsetneqq(\p\cap A)[T] \subset \p$ in
$A[T]$. Hence $\hh\p \geq \hh (\p \cap A)$ when $\p=(\p\cap A)[T]$ and $\hh \p
\geq \hh (\p \cap A)+1$ when $\p \supsetneqq (\p\cap A)[T]$.
	Now, let $\hh (\p\cap A) =
r$. Then, by the dimension theorem (see \cite{Matsu} Theorem 13.6),
 $\p\cap A$ is minimal over an
ideal $\ma =(a_1,\ldots,a_r)$. Then $(\p\cap A)[T]$ is
minimal over $\ma[T]$, so $\hh (\p\cap A)[T] \leq r$. Thus, we have
$\hh \p =\hh (\p\cap A)$ in the case $1$. 

Now, assume that $(\p\cap A)[T] \subsetneqq \p$, say $f\in \p-(\p\cap
A)[T]$. We will be done if we can show that $\p$ is a minimal prime
over $\ma [T]+f A[T]$, for then $\hh \p\leq r+1$. Let $\p'$ be a prime
between these. Then $\ma\subset (\p'\cap A) \subset (\p\cap A)$, so
$(\p'\cap A) = (\p\cap A)$, since $(\p\cap A)$ is minimal prime over
$\ma$.  In particular, $(\p\cap A)[T]\subsetneqq \p'\subset\p$. By
($\ref{555}$), we have $\p=\p'$. Thus, we are done in case $2$.

Now, we prove the lemma for any ideal $I\subset A[T]$.  Let $\sqrt I =
\bigcap_1^r \p_i$, where $\p_1,\ldots,\p_r$ are minimal primes over
$I$.  Then $\sqrt {I\cap A} = \bigcap_1^r (\p_i\cap A)$. 
	The  prime ideals minimal over
$(I\cap A)$ occur among $(\p_1\cap A),\ldots,(\p_r\cap A)$. Choose $\p_i$
such that  $\hh (I\cap
A)=\hh (\p_i\cap A)$. Then $\hh (I\cap
A)= \hh (\p_i\cap A)\geq \hh \p_i -1 \geq
\hh I-1$. This proves the lemma. 
$\hfill \square$
\end{proof}

\begin{lemma}\label{004}
Let $A$ be a ring of dimension $d$.  Suppose $K\subset A[T]$ is an
 ideal such that $K+\CJ(A) A[T]=A[T]$.  Then, any maximal ideal of
 $A[T]$ containing $K$ has height $\leq d$.
\end{lemma}

\begin{proof}
Suppose $\mm\subset A[T]$ is a maximal ideal of height $d+1$. Then
$\mm\cap A$ is a maximal ideal of $A$. Hence $\mm$ contains
$\CJ(A)$. Since $K+\CJ(A)A[T]=A[T]$, it follows that $K$ is not
contained in $\mm$. This proves the lemma.
$\hfill \square$
\end{proof}
\medskip

The following result is labeled as ``Moving lemma''. Its proof 
is similar to (\cite{Bh-Raja-1}, Lemma 3.6).

\begin{lemma}\label{move}
{\rm (Moving Lemma)}
Let $A$ be a  ring of dimension $d$ and let $n$ be an integer such that
$2n \geq d+3$. Let $I$ be an ideal of $A[T]$ of 
height $n$ and let $J=I\cap A$. 
Let $\wt P$ be  a projective $A[T]$-module of rank $n$ and $f\in A[T]$.
 Suppose 
$\phi : \wt P \surj I/(I^2f)$ be a surjection. Then, we can find a lift
$\Delta \in \Hom_{A[T]}(\wt P,I)$ of $\phi$ such that the ideal
$\Delta(\wt P)=I''$ satisfies the following properties: 

(i) $I = I'' + (J^2f)$.

(ii) $I''= I\cap I'$, where $\hh I' \geq n$.

(iii) $I' + (J^2 f) = A[T]$.
\end{lemma}

\begin{proof}
Let $\Phi$ be a
lift of $\phi$. Then $I=\Phi(\wt P) + (I^2 f)$. By (\ref{002}), there
exists $b\in (I^2 f)$ such that $I = \Phi(\wt P)+(b)$.  Let
``bar'' denote reduction modulo $(J^2 f)$. Applying Eisenbud-Evans
theorem (\ref{EE}), there exist $\Phi_1 \in \wt P^*$ such that
if $N = (\Phi+ b\Phi_1)(\wt P)$, 
 then $\hh \ol N_{\ol b} \geq n$. Since $I = N+ (b)$ and $b\in I^2$,
by (\ref{002}), we get $N = I \cap K$ with $K+ (b) = A[T]$. 
We claim that $\ol K = \ol {A[T]}$.

Assume otherwise, i.e. $\ol K$ is a proper ideal of $\ol {A[T]}$.
Since $b$ is a multiple of $f$, $K+(f)=A[T]$. Hence $\hh \ol K_{\ol f}
=\hh \ol K = \hh \ol K_{\ol b} = \hh \ol N_{\ol b} \geq n$. Therefore,
\medskip

$n \leq \hh \ol K_{\ol f} \leq \dim (\ol {A[T]}_{\ol f}) =  \dim 
((A/J^2)[T,f^{-1}])$

$\leq \dim (A/J) +1\leq 
\dim A - \hh J+1$ (by (\ref{lem1})) 

$\leq d-(n-1)+1$ (since $\hh J \geq n-1$, by (\ref{ht}))

$\leq n-1$ (since $2n\geq d+3$).
\medskip

This is a contradiction. Hence
$\ol K = \ol {A[T]}$ and $\ol N = \ol I$, i.e. $K+(J^2f)=A[T]$ and 
$I = N + (J^2 f)$. This proves the claim. 

Write $\Psi = \Phi +b \Phi_1$. Then $\Psi$ is also a lift of $\phi$.
We have $I = \Psi(\wt P)+(J^2f)$.
There exists $c \in (J^2 f)$ such that $I = \Psi(\wt P)+(c)$.
Again, applying (\ref{EE}), there exists $\Psi_1 \in \wt P^*$ such that
if $I''=(\Psi+c \Psi_1)(\wt P)$, then $\hh I_c'' \geq n$. 
	
Write $\Delta=\Psi+ c\Psi_1$. Then $\Delta$ is also a lift of $\phi$.
We have $I'' = \Delta(\wt P)$ and $I= \Delta (\wt P)+(c)$. By
(\ref{002}), we get $I'' = I\cap I'$ with $I' +(c)=A[T]$ and $\hh I' \geq
n$.

Thus, we have $(1)~ I= I'' +(J^2 f)$, 
$(2)~ I''= I \cap I'$, where $\hh I' \geq n$ and 
$(3)~ I'+(J^2 f)=A[T]$.
This proves the result.
$\hfill \square$
\end{proof}

\begin{lemma}\label{011}
Let $C$ be a ring with $\dim (C/\CJ(C)) = r$ and
let $P$ be a projective $C$-module of rank $m\geq r+1$. 
Let $I$ and $L$ be ideals of $C$ such that $L\subset I^2$.
Let $ \phi : P \surj I/L$ be a surjection. Then $ \phi$ can be
lifted to a surjection $\Psi : P \surj I$. 
\end{lemma}

\begin{proof}
Let $\Psi : P \ra I$ be a lift of $\phi$. Then $\Psi(P) + L= I$.
Since $L\subset I^2$, by (\ref{002}), there exists $e\in L$ such that
$\Psi(P) + (e)=I$.

Let ``tilde'' denote reduction modulo $\CJ(C)$. 
Then $\wt \Psi (\wt P) +(\wt e) = \wt I$. Applying Eisenbud-Evans theorem
(\ref{EE}) to the element $(\wt \Psi,\wt e)$ of $\wt {P^*} \op \wt C$,
we see that there exists $\Theta \in P^*$ such that if
$K=(\Psi+e\,\Theta)(P)$, then $\hh \wt K_{\wt e} \geq m$. As $\dim \wt C
=r\leq m-1$, we have $\wt K_{\wt e}=\wt C_{\wt e}$.
 Hence ${\wt e\,}^l \in \wt K$ for some positive integer $l$.
Since $\wt K + (\wt e)=\wt I$ and  
$e\in L\subset I^2$, by (\ref{010}), $\wt K=\wt I$.
Since $e\in L$, the element $\Psi + e\, \Theta$ is also a lift of
$\phi$. Hence, replacing $\Psi$ by $\Psi + e\Theta$, we can assume
that $\wt {\Psi(P)} = \wt I$ i.e. $\wt \Psi : \wt P \surj \wt I$ is a
surjection. Therefore, since $\wt I = (I+\CJ(C))/\CJ(C)=I/(I\cap \CJ(C))$,
 we have $\Psi(P) + (I\cap \CJ(C)) = I$. We also have $\Psi(P)+ L = I$.
Therefore, since $L\subset I^2$, by (\ref{010}), $\Psi(P)=I$.
$\hfill \square$
\end{proof}
\medskip

As a consequence of (\ref{011}), we have the following result.

\begin{lemma}\label{mand1}
Let $A$ be a ring with $\dim( A/\CJ(A))=r$.  Let $I$ and $L$
be ideals of $A[T]$ such that $L\subset I^2$ and $L$ contains a monic
polynomial. Let $P'$ be a projective $A[T]$-module of rank
$m \geq r+1$. Let $\phi: P'\op A[T] \surj I/L$ be a surjection. Then, we can
lift $\phi$ to a surjection $\Phi : P'\op A[T] \surj I$ with $\Phi(0,1)$ a
monic polynomial.
\end{lemma}

\begin{proof}
Let $\Phi'=(\Theta,g(T))$ be a lift of $\phi$. Let $f(T)\in L$ be a
monic polynomial. By adding some large power of $f(T)$
to $g(T)$, we can assume that the lift $\Phi'=(\Theta,g(T))$
of $\phi$ is such that $g(T)$ is a monic polynomial.
 Let $C=A[T]/(g(T))$. Since $A \inj C$ is an integral extension, we
have $\CJ(A) = \CJ(C) \cap A$ and hence $A/\CJ(A) \inj C/\CJ(C)
$ is also an integral extension. Therefore, $\dim (C/\CJ(C)) =r$.
 
Let ``bar'' denote reduction modulo $(g(T))$. Then, $\Theta$ induces a
surjection $\alpha : \ol {P'}  \surj \ol I/\ol L$, which, by (\ref{011}),
can be lifted to a surjection from $\ol {P'}$ to $\ol I$. 
Therefore, there exists a map $\Gamma : P' \ra I$ such that
$\Gamma(P')+ (g(T))=I$ and $(\Theta-\Gamma)(P')=K \subset L+(g(T))$.
Hence $\Theta - \Gamma \in K{P'}^*$. This shows that $\Theta -
\Gamma = \Theta_1 +g(T)\,\Gamma_1$ for some $\Theta_1 \in L{P'}^*$ and
$\Gamma_1 \in {P'}^*$. 

Let
$\Phi_1 = \Gamma + g(T)\, \Gamma_1$ and let $\Phi = (\Phi_1,g(T))$.
Then, $\Phi(P'\op A[T]) = \Phi_1(P')+(g(T)) = \Gamma(P')+(g(T)) = I$.
Thus $\Phi : P' \op A[T] \surj I$ is a surjection. Moreover, 
 $\Phi(0,1) = g(T)$ is a monic polynomial. 
Since $\Phi' - \Phi =
(\Theta -\Phi_1,0)=(\Theta_1,0)$, where $\Theta_1 \in L{P'}^*$ and $\Phi'$
is a
lift of $\phi$, we see that $\Phi$ is a (surjective) lift of $\phi$.
$\hfill \square$
\end{proof}

\begin{lemma}\label{0sub}
Let $A$ be a ring of dimension $d$ and let $I,I_1\subset A[T]$ be two
comaximal ideals of height $n$, where $2n \geq d+3$. Let $P=P_1\op A$
be a projective $A$-module of rank $n$. Assume $J=I\cap A\subset
\CJ(A)$ and $I_1+(J^2T)=A[T]$. Let $\Phi : P[T] \surj I\cap I_1$ and
$\Psi : P[T] \surj I_1$ be two surjections with $\Phi\ot
A[T]/I_1=\Psi\ot A[T]/I_1$. Then, there exists a surjection $\Lambda : P[T]
\surj I$ such that $ (\Phi-\Lambda) (P[T]) \subset (I^2T)$.
\end{lemma}
 
\begin{proof}
We first note that, to prove the lemma, we can replace $\Phi$ and
$\Psi$ by $\Phi\,\Delta$ and $\Psi\, \Delta$, where
$\Delta$ is an automorphism of $P[T]$. 

Let $\Psi = (\Psi_1,f)$. Let ``bar'' denote reduction modulo $(J^2T)$
and let $D = A[T]/(J^2T)$. Since $I_1 + (J^2T)=A[T]$, it follows that
$(\ol \Psi_1,\ol f)\in \Um(\ol {P_1[T]^*} \op D)$. 
Since $J\subset \CJ(A)$, we have $JD \subset \CJ(D)$. Moreover, 
$D/JD=(A/J)[T]$ and $\hh J \geq n-1$, by (\ref{ht}). Hence
$\dim (A/J) \leq \dim A -\hh J\leq d-(n-1) \leq n-2$. 
Therefore, since rank $P_1 =n-1$, by (\cite{P}, Corollary 2), 
$\ol {P_1[T]}$ has a unimodular element. By Lindel's result 
(\ref{lindel}), $E(\ol {P_1[T]^*} \op D)$ acts
transitively on $\Um(\ol {P_1[T]^*} \op D)$ 
and by (\ref{trans}), any element of $E(\ol {P_1[T]^*} \op D)$ can
be lifted to an automorphism of $P_1[T]\op A[T]$. 
Putting above facts together, we can assume, replacing $(\Psi_1,f)$ by
$(\Psi_1,f)\Delta$ ($\Delta$: suitable automorphism of $P[T]$) if 
necessary, that $\Psi_1(P_1[T]) + (J^2T)A[T] = A[T]$ and $f\in (J^2T)$.
Moreover, applying Eisenbud-Evans theorem (\ref{EE}), we can assume,
that $\hh (\Psi_1(P_1[T])) =n-1$. 

Since $J\subset \CJ(A)$ and $\Psi_1(P_1[T])+(J^2T)=A[T]$, we have
$\Psi_1(P_1[T])+\CJ(A)A[T]=A[T]$ and therefore, by (\ref{004}), 
any maximal ideal
of $A[T]$ containing $\Psi_1(P_1[T])$ is of height $\leq d$. 
	Hence, by (\ref{lem1}), $\dim
(A[T]/\Psi_1(P_1[T])) \leq d - \hh (\Psi_1(P_1[T]))
\leq d-(n-1) \leq n-2$.
	Hence, applying (\ref{sub}),
we get a surjection $\Lambda : P[T] \surj I$ such that
$(\Phi-\Lambda) (P[T]) \subset (I^2T)$.
$\hfill \square$
\end{proof}

\begin{lemma}\label{021}
Let $A$ be a  ring of dimension $d$ and let $n$ be an integer such
that $2n \geq d+3$. Let $I$ be an ideal of $A[T]$ of height $n$ such that
$I+\CJ(A)A[T]=A[T]$.
Assume that $\hh \CJ(A)\geq n-1$. Let $P$ be a projective $A$-module
of rank $n$ and let  $\phi : P[T] \surj I/I^2$ be a surjection. If
the surjection $\phi\ot A(T) : P(T)
\surj IA(T)/I^2A(T)$ can be lifted to a surjection from
$P(T)$ to $IA(T)$, then $\phi$ can be lifted to a
surjection $\Phi : P[T] \surj I$.
\end{lemma}

\begin{proof}
Recall that $A(T)$ denote the ring obtained from $A[T]$ by inverting all
monic polynomials and $P(T)=P[T]\ot A(T)$.
It is easy to see that, under the hypothesis of the lemma, there
exists a monic polynomial $f(T)\in A[T]$ and a surjection 
$\Phi' : P[T]_f \surj I_f$ such that $\Phi'$ is a lift of $\phi_f$.
 Since $I+\CJ(A)A[T] = A[T]$, $I$ is not 
contained in any maximal ideal of $A[T]$ which
contains a monic polynomial and hence $f(T)$ is a unit modulo $I$. 

Since $\dim (A/\CJ(A))\leq \dim A - \hh \CJ(A) \leq d-(n-1)\leq n-2$, 
by Serre's result (\ref{Serre}), $P$ has a free
direct summand of rank 2, i.e. $P=Q \op A^2$.

For the sake of simplicity of notation, we write $R$ for $A[T]$,
$\wt Q$ for $Q[T]$ and $\wt P$ for $P[T]$. Since 
$\Phi' \in \Hom_{R_f}({\wt P}_f,I_f)$, there
exists a positive even integer $N$ such that  
$\Phi'' =f^N \Phi' \in \Hom_{R}(\wt P,I)$. It is easy to see, by the 
very construction of $\Phi''$, that the induced map ${\Phi''}_f$ from $\wt
P_f$ to $I_f$ is a surjection. Since $f$ is a unit modulo $I$, the
canonical map $R/I \ra R_f/I_f$ is an isomorphism and hence 
$I/I^2 = I_f/{I_f}^2$. Putting these
facts together, we see that $\phi'' = \Phi'' \ot R/I : \wt P \surj
I/I^2$ is a surjection. Moreover, $\phi''=f^N \phi$.

\paragraph{Claim:} $\phi'' : \wt P \surj I/I^2$ can be lifted to a
surjection from $\wt P$ to $I$. 
\medskip
 
Assume the claim. Let $\gL : \wt P \surj I$ be a lift of $\phi''$.
Write $D=R/(f(T))$. Since $(f(T))+I=R$ and $\Lambda(\wt P) = I$,
 $\Lambda \ot D$ is a unimodular element of ${\wt P}^*\ot D$.  Let
 $\Lambda = (\gl,d_1,d_2)$, where $\gl\in \Hom_R(\wt Q,R)$ and
 $d_1,d_2\in R$.

Since $f(T)$ is monic, $A \hookrightarrow D$ is an integral extension
and hence $A/\CJ(A) \hookrightarrow D/\CJ(D)$ is also an integral
extension. Hence $\dim (D/\CJ(D)) =\dim
(A/\CJ(A)) \leq n-2$.
Therefore, in view of Bass' result (\ref{1bass}), 
the unimodular element $(\gl,d_1,d_2)\ot D$  can be taken to 
$(0,0,1)$ by an element of $E({\wt P}^* \ot D)$. By (\ref{trans}),
every element of $E({\wt P}^* \ot D)$ can be lifted to an automorphism of
${\wt P}^*$. Moreover, since $I+(f)=R$, a lift can be chosen to be an
automorphism of ${\wt P}^*$ which is identity modulo $I$. 

The upshot of the above discussion is that there exists an
automorphism $\Omega$ of $\wt P$ such that $\Omega$ is identity
modulo $I$ and $\Omega^*(\Lambda)=\Lambda \,\Omega =(0,0,1)$ modulo $(f(T))$.
Therefore, replacing $\Lambda$ by $\Lambda\,\Omega$,
 we can assume that $\Lambda=(\gl,d_1,d_2)$ with $1-d_2 \in (f(T))$. 

Recall that our aim is to lift the surjection $\phi : \wt P \surj I/I^2$
to a surjection $\Phi : \wt P \surj I$. Recall also that the
 surjection $\Lambda : \wt P \surj I$ is a lift of 
 $f^N \phi : \wt P \surj I/I^2$. 
 
Let $g\in R$ be such that $fg=1$ modulo $(d_2)$ and hence modulo $I$. 
Let $\ma = (g^N d_1,d_2)$. Then, since $N$ is even, by (\ref{krusem}), 
$\ma = (e_1,e_2)$ with $e_1-g^N d_1 \in {\ma}^2$ and $e_2 - g^N d_2
\in {\ma}^2$. Since $\Lambda=(\gl,d_1,d_2)$, $\Lambda(\wt P)=I$ and
$Rg+R d_2=R$, we see that
$$I=\gl(\wt Q)+ (d_1,d_2)= g^N \gl (\wt Q) + (g^Nd_1,d_2)= g^N\gl
(\wt Q) + (e_1,e_2).$$
Let $\Phi = (g^N \gl,e_1,e_2) \in \Hom_R(\wt P,I)$. From the above
equality, we see that $\Phi : \wt P \surj I$ is a surjection.
Moreover, since $1-fg\in I$, $\Phi \ot R/I = g^N \Lambda \ot R/I$ and
$\Lambda \ot R/I = f^N \phi \ot R/I$, $\Phi$ is a (surjective)
lift of $\phi$. This proves the lemma.
 
\paragraph{Proof of the claim:}
Recall that $\Phi'' : \wt P \ra I$ such that the induced map ${\Phi''}_f :
\wt P_f \surj I_f$ is a surjection and $\phi'' = \Phi''\ot R/I : \wt P \surj
I/I^2$.

We first note that if $\Delta$ is an automorphism of $\wt P$ and if
the surjection $\phi'' \Delta : \wt P \surj I/I^2$ has a surjective
lift from $\wt P$ to $I$, then so also has $\phi''$. We also note that,
by (\ref{trans}), any element of $E(\wt P/I\wt P)$ can be lifted to
an automorphism of $\wt P$. Keeping these facts in mind, we proceed
to prove the claim.

By (\ref{m-phil}), there exists $\Delta_1 \in E(\wt P_f)$ such that
$(1)~ \Psi={\Delta_1}^*(\Phi'') \in \Hom_R(\wt P,I)$ and  $(2)~ \Psi(\wt
P)$ is an ideal of $R$ of height $n$, where ${\Delta_1}^*$ is an
element of $E({\wt P_f}^*)$ induced from $\Delta_1$.

Since $\Psi_f(\wt P_f)=I_f$ and $f$ is a unit modulo $I$, we have 
 $I=\Psi(\wt P)+I^2$. Hence, by (\ref{002}), $\Psi(\wt P)=I_1= I\cap
I'$, where $I'+I=R$. Since $(I_1)_f = I_f$, ${I'}_f=R_f$ and hence $I'$
contains a monic polynomial $f^r$ for some positive integer $r$. 

Since $\Delta_1 \in E(\wt P_f)$, $\ol \Delta = \Delta_1 \ot R_f/I_f 
\in E(\wt P_f/I_f \wt P_f)$.
 Since $\wt P/I\wt P= \wt P_f/I_f \wt P_f$, we can regard $\ol
\Delta$ as an element of $E(\wt P/I\wt P)$. By (\ref{trans}), $\ol
\Delta$ can be lifted to an automorphism $\Delta$ of $\wt P$.

The map $\Psi : \wt P \surj I\cap I'$ induces a surjection $\psi :
\wt P \surj I/I^2$ and it is easy to see that $\psi = \phi'' \Delta$.
Therefore, to prove the claim, it is enough to show that $\psi$ can
be lifted to a surjection from $\wt P$ to $I$. If $I'=R$, then
obviously $\Psi$ is a required surjective lift of $\psi$. Hence, we
assume that $I'$ is an ideal of height $n$. 

The map $\Psi : \wt P \surj I \cap I'$ 
induces a surjection 
$\psi' : \wt P \surj I'/{I'}^2$. Recall that $\wt P =\wt Q \op
R^2$. Therefore, since  
$I'$ contains $f^r$; a monic polynomial and $\dim (A/\CJ(A))
\leq n-2$, by (\ref{mand1}),
$\psi'$ can be lifted to  a surjection $\Psi' (=(\Gamma,h_1,h_2))
 : \wt P \surj I'$, where $\Gamma \in {\wt Q}^*$, $h_1,h_2 \in R=A[T]$
and $h_1$ is monic. Moreover, if necessary, by (\ref{EE}), 
we can replace $\Gamma$ by $\Gamma+ {h_2}^2\, \Gamma_1$ for suitable
$\Gamma_1 \in {\wt Q}^*$ and assume that $\hh K =n-1$, where
$K=\Gamma(\wt Q) + Rh_1$. Let $\ol R = R/K$ and $\ol A = A/(K\cap
A)$. Then $\ol A \inj \ol R$ is an integral extension and hence 
$\dim (\ol R/\CJ(\ol R)) = \dim (\ol A/\CJ(\ol A)) \leq \dim (A/\CJ(A))
\leq n-2$. 

Let $P_1 = \wt Q\op R$. Then $\wt P=P_1 \op R$ and $K=\Psi'(P_1)$.
Since $K$ contains a monic polynomial $h_1$, $K+I^2=R$. Moreover,   
surjections $\Psi: \wt P \surj I\cap I'$ and $\Psi' : \wt P
\surj I'$ are such that $\Psi\ot R/I'= \Psi' \ot R/I'$. Therefore,
since $\ol R = R/K$ and $\dim (\ol R/\CJ(\ol R)) < n-1$, 
by (\ref{sub}), there exists a surjection
$\Lambda_1 : \wt P \surj I$ with
$\Lambda_1 \ot R/I= \Psi \ot R/I = \psi$. Therefore, $\Lambda = \Lambda_1
\, \Delta^{-1} : \wt P \surj I$ is a lift of $\phi''$.
Thus, the proof of the claim is complete. 
$\hfill \square$
\end{proof}

\begin{remark}
The above result has been proved in (\cite{Das}, Lemma 3.6) in case
$A$ is semi-local and $n=d \geq 3$.
\end{remark}

The following result is due to Bhatwadekar and Raja Sridharan
(\cite{Bh-Raja-1}, Lemma 3.5).

\begin{lemma}\label{3.5}
Let $A$ be a regular domain containing a field $k$, $I\subset A[T]$ an
ideal, $J= A\cap I$ and $B=A_{1+J}$. Let $P$ be a 
projective $A$-module and let $\ol \phi : P[T] \surj I/(I^2T)$ be a
surjective map. Suppose there exists a surjection $\theta : P_{1+J}[T]
\surj I_{1+J}$ such that $\theta$ is a lift of $\ol \phi \ot
B$. Then, there exists a surjection $\Phi : P[T] \surj I$ such that
$\Phi$ is a lift of $\ol \phi$.
\end{lemma}

The following result is very crucial for the proof of our main result
(\ref{general}).

\begin{proposition}\label{monic}
Let $A$ be a regular domain of dimension $d$ containing a field $k$ 
and let $n$ be an integer such that  
 $2n\geq d+3$. Let $I$ be an ideal of $A[T]$ of height $n$. Let 
 $P$ be a projective $A$-module of rank $n$ and let
 $\psi : P[T] \surj I/(I^2T)$ be a surjection. Suppose there exists a 
surjection  $\Psi' : P[T]\ot A(T) \surj IA(T)$
which is a lift of $\psi \ot A(T)$. Then, we can lift $\psi$ to a
surjection $\Psi : P[T] \surj I$.
\end{proposition}

\begin{proof}
In view of (\ref{3.5}), we can assume that $J=I\cap A\subset \CJ(A)$.
Hence $\hh \CJ(A) \geq n-1$, by (\ref{ht}) and $\dim (A/\CJ(A)) 
\leq d- (n-1) \leq n-2$. Therefore, by Serre's result
(\ref{Serre}), we can assume that $P$ has a unimodular element
i.e. $P=P_1 \op A$.

Applying Moving lemma (\ref{move})
 for the surjection $\psi : P[T] \surj I/(I^2T)$, we get
 a lift $\Theta \in \Hom_{A[T]}(P[T],I)$ of $\psi$ such that the ideal
$\Theta(P[T]) = I''$ satisfies the following properties:

$(i)~ I= I'' +(J^2T)$.

$(ii)~ I''=I\cap I'$,  where $I'$ is an ideal of height $n$. 

$(iii)~ I' +( J^2T) = A[T]$.
\medskip

The surjection
$\Theta : P[T] \surj I\cap I'$ induces surjections 
$\phi : P[T]/I'P[T] \surj I'/{I'}^2$ and
$\Theta \ot A(T)
: P(T) \surj (I\cap I')A(T)$ such that 
$\Psi' \ot A(T)/IA(T) = (\Theta \ot A(T)) \ot
A(T)/IA(T).$ 

Since $\dim A(T)=d$ and $I,I'$ are two comaximal ideals of height $n$,
where $2n\geq d+3$, applying Subtraction principle (\ref{subtraction})
to surjections $\Psi'$ and $\Theta \ot A(T)$, we get a surjection
$\Phi' : P(T) \surj I'A(T)$ such that $\Phi' \ot A(T)/I' A(T) =
\phi \ot A(T).$

Since $I'+ \CJ(A)=A[T]$ and $\phi \ot A(T)$ has a
surjective lift, namely, $\Phi' : P(T) \surj I'A(T)$, by (\ref{021}),
there exists a surjection $\Phi : P[T] \surj I'$ which is a lift of $\phi$.
 
Thus, we have  surjections $\Phi : P[T] \surj I'$ and $\Theta : P[T] \surj 
I\cap I'$ such that $\Phi \ot A[T]/I' = \phi = \Theta \ot A[T]/I'$.
Hence, as $I' + (J^2T)=A[T]$ and $J\subset \CJ(A)$, by
(\ref{0sub}), there exists a surjection $\Psi : P[T] \surj I$
such that $(\Psi-\Theta)(P[T])\subset (I^2T)$. Since
$\Theta$ is a lift of $\psi$, we are through.
$\hfill \square$
\end{proof}

\begin{remark}
For $n=d$, the above proposition has been already proved in (\cite{Das},
Theorem 4.7) in the case $A$ is an arbitrary ring
containing a field of characteristic $0$.  
\end{remark}

As an application of (\ref{monic}), we prove the following ``Subtraction
principle'' for polynomial algebra.

\begin{corollary}\label{subtract}
Let $A$ be a regular domain of dimension $d$
containing an infinite field $k$ and let $n$ be an integer such that 
$2n \geq d+3$. Let $P=  P_1 \op A$
be a projective $A$-module of rank $n$ and let
$I, I' \subset A[T]$ be two
comaximal ideals of height $n$. Let
$\Gamma : P[T] \surj I$ and $\Theta : P[T] \surj I\cap I^\prime$ be
surjections such that $\Gamma \ot A[T]/I = \Theta\ot A[T]/I$. Then,
there exists a surjection $\Psi : P[T] \surj I^\prime$ such that $\Psi \ot
A[T]/I^\prime = \Theta\ot A[T]/I^\prime$.
\end{corollary}

\begin{remark}
Since $\dim A[T] =d+1$, if $2n \geq d+4$, then, we can appeal to 
(\ref{subtraction}) for the proof. So, we need to prove the result only in the 
case $2n=d+3$. However, the proof given below in this case works equally 
well for $2n > d+3$ and hence, allows us to give a unified treatment.  
\end{remark}

\begin{proof} 
Let $K=I\cap I^\prime$. Then, since $k$ is infinite, there exists a 
$\gl\in k$ such that $K(\gl)=A$ or $K(\gl)$ has height $n$
(\cite{Bh-Raja-1}, Lemma 3.3). Therefore, 
replacing $T$ by $T-\gl$, if necessary, we assume that $K(0)=A$ or 
$\hh K(0)=n$. 

Note that $\Theta$ induces a surjection $\ol \theta : 
P[T] \surj I'/{I'}^2$. 
We first show that $\ol \theta$ can be lifted to a surjection from $P[T]$
to
$I'/({I'}^2T)$.
 
{\it Case 1.} If $I^\prime(0)=A$, then, since $P=P_1 \op A$, we
can lift  $\ol \theta$ to a surjection
$\phi : P[T] \surj I^\prime/({I^\prime}^2T)$. 

{\it Case 2.} Assume that 
$\hh I'(0)=n$. The map $\Theta$ induces a surjection $\Theta(0) : P
\surj K(0)(=I(0) \cap I'(0))$. If $I(0)=A$, then $K(0)=I'(0)$ and
therefore it is easy to see that $\Theta(0)$ and $\ol \theta$ will
patch up to give a surjection $\psi : P[T] \surj I'/({I'}^2T)$ which
is a lift of $\ol \theta$.
Now, if $\hh I(0)=n$, then, since   $\Gamma \ot A[T]/I = \Theta\ot
A[T]/I$,  
we can apply the Subtraction principle (\ref{subtraction}) to the surjections
$\Gamma(0) : P \surj I(0)$ and $\Theta(0) : P \surj I(0) \cap I'(0)$ to 
conclude that there is a surjection 
$\gf : P \surj I'(0)$ such that $\gf\ot A/I^\prime(0) = 
\Theta(0) \ot A/I^\prime(0)$. Hence, as before, we see that $\ol
\theta$ and $\gf$ will patch up to give a surjection $\psi : P[T]
\surj I'/({I'}^2T)$ which is a lift of $\ol \theta$. 

In view of (\ref{monic}), to show that there exists a surjection $\Psi
: P[T] \surj I'$ such that $\Psi \ot A[T]/I' = \ol \theta = \Theta
\ot A[T]/I'$, it is enough to show that $\psi \ot A(T)$ has a
surjective lift from $P(T)$ to $I'A(T)$.
 
The surjections $\Gamma$ and $\Theta$ induces surjections $\Gamma \ot
A(T) : P(T) \surj IA(T)$ and $\Theta \ot A(T) : P(T) \surj (I\cap
I')A(T)$ respectively with the property 
$$(\Gamma \ot A(T)) \ot
A(T)/IA(T) = (\Theta \ot A(T)) \ot A(T)/IA(T).$$ 
	Therefore, by Subtraction principle 
(\ref{subtraction}), there exists a surjection
$\Psi' : P(T) \surj I^\prime A(T)$ with the property  
$\Psi' \ot A(T)/I^\prime A(T) = 
(\Theta\ot A(T)) \ot A(T)/I^\prime A(T).$ 

Since, $(\Theta \ot A(T))
\ot A(T)/I' A(T) = \psi \ot A(T)$, we are through. 
$\hfill \square$
\end{proof}
\medskip

Recall that ring $A$ is called {\it essentially of finite type over a
field $k$}, if $A$ is a localization of an affine algebra over $k$.

Now, we prove our main result of this thesis.

\begin{theorem}\label{general}
Let $k$ be an infinite perfect field and let 
$A$ be a regular domain of dimension $d$ which is 
essentially of finite type over 
$k$. Let $n$ be an integer such that $2n \geq
d+3$. Let $I\subset A[T]$ be an ideal of height $n$ and let $P$ be a 
projective $A$-module of rank $n$.
Assume that we are given a surjection $$\phi : P[T]\surj
I/(I^2T).$$ Then, there exists a surjection
$$\Phi: P[T]\surj I$$ such that $\Phi$ is a lift of $\phi$.
\end{theorem}

\begin{remark}
We first say a few words about the method of the proof. The
essential ideas are contained in the case where $P=A^n$ is free. To
simplify the notation, we denote the ring $A[T]$ by $R$.

Following an idea of Quillen (see \cite{Quillen}), we show that the
collection of elements $s\in A$ such that the surjection
$\phi_s=\phi\ot R_s$ can be lifted to a surjection $\Psi : {R_s}^n
\surj I_s$ is an ideal of $A$. This ideal, in view of the result of
Mandal and Varma (the local case), is not contained in any maximal ideal
of $A$ and hence contains $1$.  Therefore, we are through.

Denote this collection by $\CS$. It is easy to see that if $s\in \CS$
and $a\in A$, then $as\in \CS$. Hence $\CS$ will be an ideal
if we show that for $s,t\in \CS$, $s+t \in \CS$. 
As in \cite{Quillen}, by replacing $A$ by $A_{s+t}$, we may 
assume that $s+t=1$. Since $A$ is regular, 
if some power of $s$ is in $I$,
then, by using Quillen's splitting lemma for an automorphism of
${R_{st}}^n$ which is isotopic to identity, one can easily show that $1=s+t
\in \CS$ (for example see \cite{Bh-Raja-1}, Lemma 3.5).
 The crux of the proof is to reduce the problem to this
case. We indicate in brief how this reduction is achieved. First we
digress a bit.

The surjection $\phi : R^n \surj I/(I^2T)$  can be lifted to
$\Phi' : R^n \surj I\cap I'$, where $I'$ is an ideal of $R$ of
height $n$ comaximal with $I$ (we say $I'$ is 
residual to $I$ with respect to $\phi$). A ``Subtraction principle''
(see Theorem \ref{subtraction} and Corollary \ref{subtract}) says
that if the surjection (induced by $\Phi'$) $\phi_1 : R^n \surj
I'/({I'}^2 T)$ has a surjective lift from $R^n$ to $I'$, then
$\phi$ can be lifted to a surjection $\Phi: R^n \surj I$.

Now, using the fact that $t=1-s\in \CS$, we first show the existence of
an ideal $I_1$ which is residual to $I$ with respect to $\phi$ and
satisfying the additional property that $I_1$ 
is comaximal with $Rs$. Then, using the fact that 
$s\in \CS$, we show that there
exists an ideal $I_2$ which contains a power of $s$ and is residual
to $I_1$. Thus, the desired reduction is achieved.
  
Since the problem is solved for $I_2$, applying 
``Subtraction principle'', the problem is solved for $I_1$. Applying
Subtraction principle once again, the problem is solved for $I$. This
completes the proof of Theorem \ref{general}.
\end{remark}

\paragraph{Proof of Theorem \ref{general}.}
If $I$ has  height $d+1$, then $I$ contains a monic
polynomial in $T$. Hence, by Mandal's theorem 
(\ref{mand}), we are through. Therefore,
we always assume that $n\leq d$ and hence, the inequality $2n \geq
d+3$ would imply that $d\geq 3$.

We first assume that $A$ is local. In this case, if $n\geq 4$ and
$I(0)=A$ or $I(0)$ is a complete intersection ideal of height $n$,
then, by Mandal-Varma theorem (\ref{M-V}), we are through. It is easy to see
that in the case $I(0)=A$, (\ref{M-V}) is valid even if $\hh I =\dim
A=3$. To complete the proof in the case $A$ is local we proceed as
follows. 

Let $J=I\cap A$. By Moving lemma (\ref{move}),
the surjection $\phi : P[T]\surj I/(I^2T)$  has a
lift $\Phi' \in \Hom_{A[T]}(P[T],I)$ such that the ideal $\Phi'(P[T])
=I''$ satisfies the following properties:

$(i)~ I''+(J^2T)=I$.

$(ii)~ I'' = I\cap I'$, where $I'$ is an ideal of height $\geq n$.

$(iii)~ I'+(J^2T)=A[T]$.
\medskip
 
Since $I'$ is locally generated by $n$ elements, 
if $\hh I'>n$, then $I'=A[T]$ and we are through. So assume that $\hh
I' = n$. The surjection $\Phi' : P[T] \surj
I'' (=I\cap I')$ induces a surjection $\psi' : P[T] \surj
I'/{I'}^2$.
Since $I'+(J^2T)=A[T]$, $I'(0)=A$. Hence, as $P$ is free,
$\psi'$ can be lifted to a surjection $\psi : P[T]
\surj I'/({I'}^2 T)$. Now, as $I'(0)=A$, by (\ref{M-V}), the
surjection $\psi$
can be lifted to a surjection $\Psi : P[T] \surj I'$. 
Thus, we have surjections $\Phi' : P[T] \surj I\cap I'$ and $\Psi : P[T]
\surj I'$ such that $\Phi' \ot A[T]/I' = \Psi\ot A[T]/I'$. Therefore,
since $I'+(J^2T)=A[T]$ and $A$ is local,
 by (\ref{0sub}), there exists a surjection
$\Phi : P[T] \surj I$ such that $(\Phi - \Phi')(P[T]) \subset (I^2T)$.
Since $\Phi'$ is a lift of $\phi$, we are through.

Now, we prove the theorem in the general case. 
Let $$ S= \{ s\in A \;|\; \exists \;
\Lambda : P_s[T] \surj I_s\;;\; \Lambda\; {\mbox {is a lift of}}\;
\phi \ot A_s[T] \,\}. $$  Our aim is to prove that $1\in S$. 
Note that if $t\in S$ and $a\in A$, then $at\in S$.
Moreover, since the theorem is proved in the local
case, it is easy to see that 
 for every maximal ideal $\mm$ of $A$, there exists $s\in
A-\mm$ such that $P_s$ is free and $s\in S$. Hence, 
 we can find $s_1,\ldots,s_r\in S$
such that $P_{s_i}$ is free and $s_1+\cdots+ s_r =1$. Therefore, 
by inducting on $r$,
it is enough to show that if $s,t\in S$ and $P_s$
is free, then $s+t\in  S$. Since, in the ring  $B=A_{s+t}$, $x+y=1$,
where $x=s/s+t$ and
$y=t/s+t$, replacing $A$ by $B$ if necessary, we are reduced to prove
that if $s,1-s=t \in S$ and $P_s$ is free, then $1\in S$.
  
The rest of the argument is devoted to the proof of this assertion.
 The proof is given in steps.

\paragraph{Step 1:}
Let $J=I\cap A$.  In view of (\ref{3.5}), replacing $A$ by $A_{1+J}$
if necessary, we assume that $J\subset \CJ(A)$. If $s$ or $t$ is a
unit in $A$, then obviously $1\in S$. So, without loss of generality,
we can assume that $s$ and $t$ are not invertible elements of
$A$. Therefore, as $J\subset \CJ(A)$, $s\notin \sqrt J$ and $t\notin
\sqrt J$.

Since $\hh I =n$, $\hh J \geq n-1$ by (\ref{ht}).
 Therefore $\dim (A/\CJ(A)) \leq d-(n-1)\leq n-2$.
 Hence, since rank $P =n$, by Serre's result
(\ref{Serre}), $P\iso Q \op A^2$.

Let $\Gamma_2 :
P_t[T] \surj I_t$ be a surjection which is a lift of $\phi\ot
A_t[T]$. Since $As+ At=A$, applying (\ref{3.2}) (with $L=(I^2T)$ and
$B=A[T]$), we get a surjection $\gamma'
: P[T] \surj I/(I^2Ts)$ which is a lift of $\phi$. Applying
(\ref{move}) to the surjection $\gamma'$,
 we get a lift  $\Gamma'\in \Hom_{A[T]} (P[T],I)$
of $\gamma'$ such that the ideal $\Gamma'(P[T])=\wt I$ satisfies the
following properties:

$(i)~ \wt I + (J^2Ts) = I$.

$(ii)~ \wt I = I\cap I_1$, where $\hh I_1 \geq n$.

$(iii)~ I_1+(J^2Ts)=A[T]$.
\medskip

As before, if $\hh I_1 > n$, then $I_1=A[T]$ and we are through. So
we assume that $\hh I_1 = n$.
The surjection $\Gamma' : P[T] \surj I\cap I_1$ induces a surjection
$\theta : P[T] \surj I_1/{I_1}^2$. 
Recall that $J\subset \CJ(A)$ and $P \iso Q \op A^2$.
Moreover, $I_1+ (J^2T) =A[T]$.
Therefore, if $\theta$ can be lifted to a
surjection $\Theta : P[T] \surj I_1$, then, by  (\ref{0sub}), $\phi$
can be lifted to a surjection $\Phi : P[T] \surj I$. 

In subsequent steps, we will show that $\theta$ has a surjective lift
$\Theta : P[T] \surj I_1$. 

\paragraph{Step 2:}
Let $\Gamma_1 : P_s[T]\surj I_s$ be a surjection which is a lift of
$\phi\ot A_s[T]$. Since the map $\Gamma' : P[T] \surj I \cap I_1$ 
is a lift of $\phi$, $\Gamma' \ot A_s[T]/I_s=\Gamma_1\ot A_s[T]/I_s$.
Therefore, applying Subtraction principle (\ref{subtract}),
 we get a surjection
$\Theta_1 : P_s[T] \surj (I_1)_s$ which is a lift of $\theta \ot
A_s[T]$.

Since $I_1+(J^2 Ts) = A[T]$, there exists an element $g\in A[T]$ such
that $1-sg \in I_1$ and the canonical map $A[T]/I_1 \ra
A_s[T]/(I_1)_s$ is an isomorphism. Therefore, as $P[T] = Q[T] \op A^2[T]
$ and $P_s[T]$ is a free $A_s[T]$-module, $Q[T]/I_1 Q[T]$ is a
stably free $A[T]/I_1$-module of rank $n-2$.
Since $J\subset \CJ(A)$, $I_1+JA[T] =A[T]$ and $\hh I_1=n$, by
(\ref{004}), any maximal ideal of $A[T]$ containing $I_1$ has height
$\leq d$. Hence 
 $\dim (A[T]/I_1) \leq d-n \leq n-3$. Hence, by Bass' result 
(\ref{1bass}), $Q[T]/I_1Q[T]$ is a free $A[T]/I_1$-module.
 
Let $N$ be a positive  even integer such that $(s^N \Theta_1)(P[T])\subset
I_1$ and let $\wt \Theta=s^N \Theta_1 \in \Hom_{A[T]} (P[T],I_1)$. Then, as
$1-sg\in I_1$, $\wt \Theta$ induces 
a surjection $\wt \theta : P[T] \surj I_1/{I_1}^2$. 
Since $N$ is even,  
if $\wt \theta$ can be lifted to a surjection $\Theta_2 : P[T] \surj
I_1$, then, by (\ref{square-lift}), there would exist a surjection
$\Theta : P[T] \surj I_1$ such that $\Theta \ot A[T]/I_1 = g^N
\Theta_2 \ot A[T]/I_1$. In that case, 
 since $1-s^N g^N\in I_1$, $A[T]/I_1 = A_s[T]/(I_1)_s$,
 $\Theta_2 \ot A[T]/I_1  = s^N \Theta_1 \ot A[T]/I_1$ and $\Theta_1$ is
a lift of $\theta$,  $\Theta$ would be a lift of $\theta$.

Thus, it is enough to show that the surjection $\wt \theta : P[T]
 \surj I_1/{I_1}^2$ can be lifted to a surjection $\Theta_2 : P[T]
  \surj I_1$.

\paragraph{Step 3:}
Recall that $\Theta_1 :P_s[T] \surj (I_1)_s$ is a surjection and  
$\wt \Theta = s^N \Theta_1 : P[T] \ra I_1$ is a lift of $\wt \theta$.
Therefore, the induced map $\wt \Theta_s : P_s[T] \surj (I_1)_s$ is
also a surjection.   
Hence, by (\ref{m-phil}), there exists
$\Delta \in E(P_s[T])$ such that if $\Delta^*(\wt \Theta)=\Lambda$, then
$(1)\;\Lambda \in P[T]^*$ and $(2)\; \gL_1(P[T])=K\subset I_1$
 is an ideal of $A[T]$ of height $n$, where $\Delta^*$ is an element
of $E(P[T]^*)$ induced by $\Delta$. Since $K_s = (I_1)_s$ and 
$A[T] \cap (I_1)_s = I_1$ (as the ideals $I_1$ and $sA[T]$
are comaximal),  we get
$K = I_1 \cap I_2$ with $(I_2)_s = A_s[T]$. Therefore, $s^r \in I_2$
and  hence $I_1 +I_2 = A[T]$, since $I_1+(s) = A[T]$.
Since $K$ is an ideal of $A[T]$ of height $n$ which is a surjective
image of $P[T]$, either $I_2=A[T]$ or $I_2$ is an ideal of height
$n$. 

Since $ A[T]/I_1 = A_s[T]/(I_1)_s$,
$P[T]/I_1P[T] = P_s[T]/I_1P_s[T]$. Hence, the element $\Delta$ of 
$E(P_s[T])$ gives rise to an
element $\ol \Delta$ of $E(P[T]/I_1P[T])$. By (\ref{trans}), there
exists an automorphism $\Delta_0$ of $P[T]$ which is a lift of $\ol
\Delta$. Let $\wt \theta\,
\ol \Delta = \lambda_1 : P[T]/I_1P[T] \surj I_1/{I_1}^2$ be a surjection.
Then, it is obvious  that if $\lambda_1$ can be lifted to a
surjection $\Lambda_1 : P[T] \surj I_1$, then $\wt \theta$ also has a
surjective lift $\Theta_2 : P[T] \surj I_1$.

\paragraph{Step 4:}
Note that $\Lambda : P[T] \surj I_1 \cap I_2$ is a surjection such
that $\Lambda \ot A[T]/I_1=\lambda_1$. Therefore, if 
$I_2=A[T]$, then we are through. Now, we
assume that $I_2$ is an ideal of $A[T]$ of height $n$.

Since $I_1(0)=A$, $\Lambda$ gives rise to a surjection $\lambda_2 :
P[T] \surj I_2/({I_2}^2T)$. If $\lambda_2$ has a surjective lift
from $P[T]$ to $I_2$, then, by Subtraction principle (\ref{subtract}),
$\lambda_1$ would
have a surjective lift $\Lambda_1 : P[T] \surj I_1$. Therefore, it is
enough to show that $\lambda_2$ can be lifted to a surjection
$\Lambda_2 : P[T] \surj I_2$.
  
Since $s^r\in I_2 \cap A$ and $t=1-s$, by
(\ref{3.5}), it is enough to show that $\lambda_2 \ot A_t[T] : P_t[T]
\surj (I_2)_t/({I_2}^2T)_t$ has a surjective lift. 
In view of (\ref{monic}), it is sufficient to prove that the surjection 
$\lambda_2 \ot A_t(T) : P_t(T) \surj I_2 A_t(T)/{I_2}^2 A_t(T)$ 
can be lifted to a surjection $\wt \Lambda_2 : P_t(T) \surj I_2 A_t(T)$.

Recall that we have a  surjection 
$\Gamma_2: P_t[T]\surj I_t$ which is a lift of $\phi \ot
A_t[T]$. Moreover, we also have surjections 
$\Gamma' : P[T] \surj I\cap I_1$, $\Lambda
: P[T] \surj I_1\cap I_2$, where $I_1$ and $I_2$ are ideals of $A[T]$
of height $n$ and an automorphism $\Delta_0$ of $P[T]$ 
 such that 

(1) $\Gamma' \ot A[T]/I = \phi$.

(2) $I_1 + (J^2Ts) = A[T]$, where $J=I\cap A \subset \CJ(A)$.

(3) $I_1+I_2=A[T]$.
 
(4) $s^N\, \Gamma' \ot A[T]/I_1 = \Lambda\, {\Delta_0}^{-1} \ot A[T]/I_1$,
where $N$ is an even integer. 
\medskip

Let $R_1 = A_t(T)$. Then, by Subtraction principle
(\ref{subtraction}),  there exists a surjection
$\Phi_1 : P[T]\ot R_1 \surj I_1R_1$ such that $\Phi_1 \ot R_1/I_1R_1 =
\Gamma' \ot R_1/I_1R_1$. Since $P[T]= Q[T] \op A[T]^2$ and
$Q[T]/I_1Q[T]$ is free,
 by (\ref{square-lift}), there exists a surjection $\Phi_2 : P[T]\ot
R_1 \surj I_1R_1$ such that $\Phi_2 \ot R_1/I_1R_1 = s^N \,\Gamma' \ot
R_1/I_1R_1 = \Lambda \,
{\Delta_0}^{-1} \ot R_1/I_1R_1$. Since $\Delta_0$ is an automorphism of
$P[T]$, there exists a surjection $\Phi_3 : P[T]\ot R_1 \surj I_1R_1$
such that $\Phi_3 \ot R_1/I_1R_1 = \Lambda \ot
R_1/I_1R_1$. Therefore, by (\ref{subtraction}), there exists a surjection 
$\wt \Lambda_2 : P[T] \ot R_1 \surj I_2 R_1$ such that $\wt \Lambda_2 \ot
R_1/I_2R_1 = \lambda_2 \ot R_1$. 

Thus, the proof of the theorem is complete. 
$\hfill \square$


\chapter{Some Auxiliary results}

In this section we prove two results. Though these results do not
have any direct bearing on the main theorem (proved in the last section),
we think that they are interesting off shoots of (\ref{mand1}) and
(\ref{subtraction}) and are of independent interest. 

First result gives a partial answer to the following question of
Roitman: 

\begin{question}
Let $A$ be a ring and let $P$ be a projective $A[T]$-module such
that $P_{f(T)}$ has a unimodular element for some monic polynomial
$f(T)$. Then, does $P$ have a unimodular element? 
\end{question}

Roitman in (\cite{Roitman}, Lemma 10) answered this question affirmatively in
the case $A$ is local.  If rank $P > \dim A$, then, by Plumstead's
result (\cite{P}, Theorem
2), $P$ has a unimodular element. In (\cite{Bh-Raja-4}, Theorem 3.4) an
affirmative answer is given to the above question in the case rank
$P= \dim A$ under the additional assumption that $A$ contains an
infinite field. In this section we settle the case (affirmatively):
$P$ is extended from $A$,
rank $P \geq (\dim A +3)/2$ and $A$ contains an infinite field.   

For the proof we need the following two lemmas which are proved in
(\cite{Bh-Raja-4}, Lemma 3.1 and Lemma 3.2 respectively).

\begin{lemma}\label{006}
Let $A$ be a ring containing an infinite field $k$ and let $\wt P$ be
a projective $A[T]$-module of rank $n$. Suppose $\wt P_{f(T)}$ has a
unimodular element for some monic polynomial $f(T)\in A[T]$. Then,
there exists a surjection from $\wt P$ to $I$, where $I\subset A[T]$
is an ideal of height $\geq n$ containing a monic polynomial.
\end{lemma}

\begin{lemma}\label{isotopic}
Let $R$ be a ring and let $Q$ be a projective $R$-module. Let $(\ga(T),f(T))
: Q[T]\op R[T] \surj R[T]$ be a surjective map with $f(T)$ monic. Let
$pr_2 : Q[T]\op R[T] \surj R[T]$ be the projection onto the second
factor. Then, there exists an automorphism $\sigma(T)$ of
$Q[T]\op R[T]$ which is isotopic to  identity 
and $pr_2\,\sigma(T) = (\ga(T),f(T))$. 
\end{lemma}

The following two results are  easy to prove. We give the proof for the sake
of completeness.

\begin{lemma}\label{ex}
Let $A$ be a ring and let $P$ be a projective $A$-module. Let $\Phi
=(\phi,f(T)) : P[T]\op A[T] \surj A[T]$ be a surjection.  
Suppose $f(T)\in A[T]$ is a monic polynomial. Then,
kernel of $\Phi$ is extended from $A$.
\end{lemma}

\begin{proof}
Let $Q=$ ker$(\Phi)$.
By (\cite{Quillen}, Theorem 1), it is enough to show that $Q_{\mm}$ is
free for every maximal ideal $\mm$ of $A$. Hence, we can assume that
$A$ is local and hence $P$ is free. 
Applying Horrock's
theorem (see \cite{Horrocks}) which says that 
``if $A$ is local ring and $\wt P$ is a projective
$A[T]$-module, then $\wt P_f$ free for $f\in A[T]$ monic implies that
$\wt P$ is
free'', it is enough to show that $Q_f$ is free. But $Q_f \iso P[T]_f$
which is free. Hence, we are through. 
$\hfill \square$
\end{proof}

\begin{lemma}\label{lem9}
Let $A$ be a ring and let $P$ be a projective
$A$-module. Let $\Phi : P[T] \surj A[T]$ and $\Psi : P[T] \surj
A[T]$ be two surjections such that $\Phi(0) = \Psi(0).$ 
Further, assume that the projective $A[T]$-modules kernel of $\Phi$ and
kernel of $\Psi$ are extended from $A$. Then, there exists an
automorphism $\Delta$ of $P[T]$ such that
$\Psi \Delta = \Phi$ and $\Delta(0) = Id$
\end{lemma}   

\begin{proof}
We first show that there exists an automorphism $\Theta$ of $P[T]$
such that $\Theta(0)= Id$ and $\Phi\,\Theta= \Phi(0) \ot
A[T]$. Let $Q=$ ker$(\Phi)$ and $L=$ ker$(\Phi(0))$.  Since
$Q$ is extended from $A$, there exists an isomorphism $\Gamma : L[T] \iso
Q$.  Now, since the rows of the following diagram 
\Com
$$
\xymatrix{ 
	0 \ar [r] & L[T] \ar [r] \ar [d]^\Gamma & P[T] \;\;\;\ar
	[r]^{\Phi(0)\ot A[T]} \ar@{-->} [d]^\gL & 
	A[T] \ar [r] \ar [d]^{Id} & 0  \\ 
	0 \ar [r] & Q \ar [r] & P[T] \ar [r]_{\Phi} &
	A[T] \ar [r] & 0
	}$$ 
are split, we can find an automorphism $\gL$ of $P[T]$ such that
the above diagram is commutative. We  have 
$\Phi\gL= \Phi(0)\ot A[T]$ and hence
$\Phi(0)\gL(0)=\Phi(0)$. Consider an automorphism 
$\Theta= \gL(\gL(0)\ot A[T])^{-1}$ of $P[T]$. 
Then $\Phi \Theta= ( \Phi(0)\ot A[T])
	(\gL(0)\ot A[T])^{-1}= 
	(\Phi(0)\ot A[T])$ and $\Theta(0)= Id$.

Similarly, we have an automorphism $\gT_1$ of $P[T]$ such that
$\Psi\gT_1=\Psi(0)\ot A[T]$ and $\gT_1(0)= Id$.
Consider the automorphism $\Delta=\gT_1\Theta^{-1}$ of
$P[T]$.  As $\Phi(0) = \Psi(0)$, we have
	$\Psi\Delta=(\Psi(0)\ot
	A[T])\Theta^{-1}=(\Phi(0)\ot A[T])
	\Theta^{-1}= \Phi$
and $\Delta(0)= Id$. This proves the lemma.  
$\hfill \square$
\end{proof}

\begin{theorem}\label{5.4}
Let $A$ be a ring of dimension $d$ containing an infinite
field $k$ and let $\wt P$ be a projective $A[T]$-module of rank $n$
which is extended from $A$,
where $2n\geq d+3$. Suppose $\wt P_{f(T)}$ has a unimodular element for
some monic polynomial $f(T)\in A[T]$. Then $\wt P$ has a unimodular
element.
\end{theorem}

\begin{proof}
By (\ref{006}), we get a surjection $\Phi : \wt P\surj I$,
where $I$ is an ideal of height $\geq n$ containing a monic
polynomial. Since $I$ is locally generated by $n$ elements, 
if $\hh I > n$, then $I=A[T]$ and hence
$\wt P$ has a unimodular element.
Hence, we assume that $\hh I=n$. 

Since $\wt P$ is extended from $A$, we write
$\wt P = P[T]$, where $P$ is a projective $A$-module of rank $n$.
Then $\Phi$ induces a surjection $\phi : P[T] \surj I/(I^2T)$ which
in its turn induces a surjection $\Phi(0) : P \surj I(0)$.

Let $J=A\cap I$. Then $\hh J \geq n-1$, by (\ref{ht}).
Since $\dim (A/J) \leq d-(n-1) \leq n-2$ and
$J\subset \CJ(A_{1+J})$, by Serre's result 
(\ref{Serre}), $P_{1+J}$ has a free direct
summand. Let $P_{1+J} = Q\op A_{1+J}$ for some projective
$A_{1+J}$-module $Q$ of rank $n-1$. Since $\dim (A/J) \leq n-2$, by
(\ref{mand1}), the surjection $\phi\ot A_{1+J}[T]$ can be lifted to
 a surjection $\Psi (= (\psi,h(T)))
 : P_{1+J}[T] (=Q[T] \op A_{1+J}[T]) \surj I_{1+J}$ with
$h(T)$ a monic polynomial.
Hence $\Phi(0)\ot A_{1+J} = \Psi(0)$. 
 
It is easy to see that there exists $a\in J$
such that if $b=1+a$, then, there exists a projective $A_b$-module $Q_1$
with the properties (i) $Q_1\ot A_{1+J} =Q$, (ii) $P_b = Q_1\op A_b$, (iii)
$\Psi: P_b[T] \surj IA_b[T]$ and (iv)  
$ \Phi(0)_b = \Psi(0)$.
Let $pr_2 : Q_1[T]\op A_b[T] \surj A_b[T]$ be the surjection defined
by $pr_2(q,x)=x$ for $q\in Q_1[T]$ and $x\in A_b[T]$. We have the
followings: 

$(1)$ Since $a\in J$, $I(0)_a=A_a$ and hence $\Phi(0)_a\ot A_a[T]$ is 
a surjection from $P_a[T]$ to $A_a[T]$.
Since $\Psi_a = (\psi,h(T))_a$ is a unimodular element
of $P_{ab}[T]^*$ with $h(T)$ monic,  by (\ref{isotopic}), 
unimodular elements $(pr_2)_a$ and 
$\Psi_a$ of $P_{ab}[T]^*$ are isotopically connected.

$(2)$ Since $h(T)$ is monic, by (\ref{ex}), kernel of $\Psi_a$ is a
projective $A_{ab}[T]$-module which is extended from $A_{ab}$.
Therefore, applying (\ref{lem9}) for the surjections $\Psi_a$ and
$\Psi(0)_a \ot A_{ab}[T]$,
there exists an automorphism $\Theta$ of $P_{ab}[T]$
such that $\Theta(0)$ is identity automorphism of $P_{ab}$ and
$\Psi_a \Theta = \Psi(0)_a\ot A_{ab}[T] = \Phi(0)_{ab} \ot A_{ab}[T]$.
By (\ref{lem8}), $\Theta$ is isotopic to identity.  
Hence, by (\ref{lem7}), 
$\Psi_a$ and $\Phi(0)_{ab} \ot A_{ab}[T]$ are isotopically
connected.

Thus, combining $(1)$ and $(2)$, the 
unimodular elements $(pr_2)_a$ and $\Phi(0)_{ab} \ot
A_{ab}[T]$ are isotopically connected.  
Therefore, there exists an automorphism $\Gamma$ of $P_{ab}[T]$ such that
$\Gamma$ is isotopic to identity and $\Phi(0)\ot A_{ab}[T]\;
 \Gamma = (pr_2)_a$.

Applying (\ref{isotopy}), we get $\Gamma = \Omega'_b\, \Omega_a$,
where $\Omega$ is an $A_b[T]$-automorphism of $P_b[T]$ and $\Omega'$
is an $A_a[T]$-automorphism of $P_a[T]$.  Hence, we have surjections
$\Delta_1=pr_2 \,\Omega^{-1} : P_b[T] \surj A_b[T]$ and
$\Delta_2=\Phi(0)\ot A_a[T]\; \Omega' : P_a[T] \surj A_a[T]$ such that
$(\Delta_1)_a = (\Delta_2)_b$. Therefore, they patch up to yield a
surjection $\Delta : P[T] \surj A[T]$. Hence $\wt P = P[T]$ has a
unimodular element. This proves the result.  $\hfill \square$
\end{proof}
\medskip

Since every projective $A[T]$-module is extended from $A$, when $A$ is
a regular ring containing a field \cite{pope}. Hence, the following
corollary is immediate from the above result.

\begin{corollary}
Let $A$ be a regular ring of dimension $d$ containing an infinite
field $k$ and let $\wt P$ be a projective $A[T]$-module of rank $n$,
where $2n\geq d+3$. Suppose $\wt P_{f(T)}$ has a unimodular element for
some monic polynomial $f(T)\in A[T]$. Then $\wt P$ has a unimodular
element.
\end{corollary}

Now, we prove our second result which is a complement of the ``Subtraction  
principle'' (\ref{subtraction}) and is labeled as ``Addition principle''.
For this result we need the following lemma
which is proved in (\cite{Bh-Raja-3}, Corollary 2.14) for $n=d$
and in (\cite{Bh-Raja-5}, Corollary 2.4) in the case $P$ is free. The
idea of the proof here is same as in (\cite{Bh-Raja-3, Bh-Raja-5}). We
give the proof for the sake of completeness.

\begin{lemma}\label{moving}
Let  $A$ be a  ring of dimension $d$ and let $P$ be a
projective $A$-module of rank $n$, where $2n\geq d+1$. Let $J\subset
A$ be an ideal of height $n$ and let $\phi :
P/JP \surj J/J^2$ be a surjection. Then, there exists an ideal
$J'\subset A$ of height $\geq n$, comaximal with $J$
 and a surjection $\Phi : P\surj J\cap J'$ such that 
	$\Phi \ot A/J = \phi $.
Further, given finitely many ideals $J_1,\ldots,J_r$ of height
$n$, $J'$ can be chosen to be comaximal with $\cap_1^r J_i$. 
\end{lemma}

\begin{proof}
Let $K=J^2\cap J_1\cap \ldots\cap J_r$. Then, by assumption,
$\hh K=n$. First, we show that the surjection $\phi$ can be lifted to a
surjection from $P/KP$ to $J/K$.

Since $K\subset J^2$, $(J/K)^2=J^2/K$. Let $\Psi
\in \Hom_{A/K}(P/KP,J/K)$ be a lift of $\phi$. Then
$\Psi(P/KP)+J^2/K=J/K$ and hence, by (\ref{002}), there exists $c\in
J^2/K$ such that $\Psi(P/KP)+(c)=J/K$. Now, applying Eisenbud-Evans
theorem (\ref{EE}), 
there exists $\Psi' \in (P/KP)^*$ such that 
$\hh N_c \geq n$, where $N=(\Psi+c\Psi')(P/KP)$. Since $\hh
K\geq n$, $\dim (A/K) \leq d-n \leq n-1$. This implies that
$N_c=(A/K)_c$. Hence $c^s \in N$ for some positive integer
$s$. Therefore, as $N+(c)=J/K$ and $c\in (J/K)^2$, we have $N=J/K$, by
(\ref{010}). Thus, as $\Psi''=\Psi+c\Psi'$ is also a lift of $\phi$, the
claim is proved.

Let $\Theta\in \Hom_A(P,J)$ be a lift of $\Psi''$. Then, as
$J/K=\Psi''(P/KP)$, we have $\Theta(P)+K=J$. By (\ref{002}), we get
$a\in K$ such that $\Theta(P)+(a)=J$. Again, applying (\ref{EE}) to
the element $(\Theta,a) \in P^*\op A$, there exists
$\Theta'\in P^*$ such that $\hh J_1=n$, where
$J_1=(\Theta+a\Theta')(P)$. 

Since $J_1+(a)=J$ and $a\in J^2$, by (\ref{002}), $J_1=J\cap J'$ and
$J'+(a)=A$. Now, setting $\Phi=\Theta+a\Theta'$, we are through.
$\hfill \square$
\end{proof}

\begin{theorem}\label{add}
{\rm (Addition Principle)}
Let $A$ be a  ring of dimension $d$. Let $J_1, J_2 \subset
A$ be two comaximal ideals of height $n$, where $2n \geq d+3$. Let
$P=Q\op A$ be a projective $A$-module of rank $n$. Let $\Phi : P\surj
J_1$ and $\Psi : P\surj J_2$ be two surjections.
Then, there exists a surjection $\Theta : P\surj J_1\cap J_2$
such that $\Phi \ot A/J_1= \Theta \ot A/J_1$ and 
$\Psi \ot A/J_2 = \Theta \ot A/J_2$.   
\end{theorem}

\begin{proof}
Let $J=J_1\cap J_2$. Since $J_1+J_2=A$, we have
$J/J^2=J_1/{J_1}^2\op J_2/{J_2}^2$. Hence $\Phi$
and $\Psi$ induces a surjection $\gamma : P\surj J/J^2$ such that
$\gamma \ot A/J_1 = \Phi \ot A/J_1$ and $\gamma \ot A/J_2 = \Psi\ot
A/J_2$. 

Applying (\ref{moving}), we get an ideal $K$ of height $n$ which is
comaximal with $J$ and a surjection $\Gamma : P\surj J\cap K$ such that
$\Gamma\ot A/J = \gamma\ot A/J$. Hence 
$\Gamma\ot A/J_1=\Phi \ot A/J_1$ and $\Gamma \ot A/J_2 = \Psi \ot A/J_2$.

Applying Subtraction principle 
(\ref{subtraction}) for the surjections $\Phi$ and
$\Gamma$, we get a surjection $\Lambda : P\surj J_2\cap K$ such that
$\Lambda \ot A/(J_2\cap K) = \Gamma\ot A/(J_2\cap K)$. Hence 
$\Lambda \ot A/J_2 = \Psi \ot A/J_2$.

Again, applying (\ref{subtraction}) for the surjections $\Psi$ and $\Lambda$,
 we get a surjection $\Delta : P\surj K$ such that
$\Delta \ot A/K = \Lambda \ot A/K$. Since $\Lambda \ot A/K = \Gamma
 \ot A/K$, we have $\Delta \ot A/K = 
\Gamma\ot A/K$.

Applying (\ref{subtraction}) for the surjections $\Delta$ and $\Gamma$, we get
a surjection $\Theta : P\surj J$ such that $\Theta \ot A/J = \Gamma\ot
A/J$. Hence $\Theta \ot A/J_1 = \Phi \ot A/J_1$ and
$\Theta \ot A/J_2 = \Psi \ot A/J_2$.  This proves the result.
  $\hfill \square$
\end{proof}
\medskip

In a similar manner, using (\ref{subtract}), we have the following
``Addition principle'' for polynomial algebra.

\begin{theorem}
Let $A$ be a regular domain of dimension $d$
containing an infinite field $k$ and let $n$ be an integer such that 
$2n \geq d+3$. Let $P=  P_1 \op A$ be a projective $A$-module of 
rank $n$ and let $I, I' \subset A[T]$ be two
comaximal ideals of height $n$. Let 
$\Gamma : P[T] \surj I$ and $\Theta : P[T] \surj I^\prime$ be two
surjections. 
Then, there exists a surjection $\Psi : P[T] \surj I \cap 
I^\prime$ such that $\Psi \ot A[T]/I = \Gamma \ot A[T]/I$ and $\Psi \ot
A[T]/I^\prime = \Theta\ot A[T]/I^\prime$. 
\end{theorem}

\begin{application}
We will end this chapter by discussing some possible applications of 
Theorem \ref{general}. Let $A$ be a regular affine domain of
dimension $d$ over an infinite perfect field $k$. Let $P$ be a projective
$A$-module of rank $n$. It is interesting to know when $P$ has a unimodular
element. By a  classical result of Serre (\cite{S1}), if $n>d$, then $P$
has a unimodular element.
It is well known that this result is not true in general if $n=d$. 
So one can ask, if one can find the obstruction for a projective
module $P$ of rank $=\dim A$ to have a unimodular element.

In (\cite{Mu1}, Theorem 3.8), Murthy proved that if $P$ is a projective
$A$-module of rank $n=d$ and $k$ is algebraically closed, then, a
necessary
and sufficient condition for $P$ to have a unimodular element
is the vanishing of its ``top Chern class" $C_n(P)$ in the Chow group
$CH_0(A)$ of zero cycles modulo rational equivalence. However, this result
of Murthy is not true if $k$ is not algebraically closed, as is evidenced
by the example of the tangent bundle of the real $2$-sphere. 

To tackle the above question when $k$ is not necessarily algebraically
closed, Nori defined  the notion of {\it Euler class group}
of $A$ (see \cite{Bh-Raja-1})
 and to any projective $A$-module $P$ of rank $=\dim A$
with trivial determinant, he attached an element of this group, called
the Euler class of $P$. Then, he asked whether non-vanishing of Euler class
of $P$ is the only obstruction for $P$ to have a unimodular element.
Proving (\ref{general}) in the case $\dim
(A[T]/I)=1$, Bhatwadekar and Raja Sridharan
 answered Nori's question in affirmative (see \cite{Bh-Raja-1}). More
precisely, they proved that a necessary and sufficient condition for
$P$ to have a unimodular element is the vanishing of the
Euler class of $P$.

Now, let $A$ be as above and $2n\geq d+3$. Then, we can define the notion of
$n^{th}$ Euler class group of $A$, denoted by $E^n(A)$
(see \cite{Bh-Raja-5}).  Let $P$ be a
projective $A$-module of rank $n$ with trivial determinant. We believe
that using (\ref{general}),
one can attach an element of $E^n(A)$
corresponding to $P$ (the Euler class of $P$)
which will detect an obstruction for $P$ to have a unimodular element.
\end{application}


\end{document}